\documentclass{amsart}

\newtheorem{theorem}{Theorem}[section]
\newtheorem{lemma}[theorem]{Lemma}
\newtheorem{proposition}[theorem]{Proposition}
\theoremstyle{definition}
\newtheorem{definition}[theorem]{Definition}

\theoremstyle{remark}

\numberwithin{equation}{section}

\usepackage{graphicx}
\usepackage{color}
\usepackage{fancyhdr}
\begin{document}
\title{elliptic boundary value problems for the stationary vacuum spacetimes}
\author{Zhongshan An}
\begin{abstract}
We develop a general method of proving the ellipticity of boundary value problems for the stationary vacuum space time, by showing that the stationary vacuum field equations are elliptic subjected to a geometrically natural collection of boundary conditions in the projection formalism. Using this we prove that the moduli space of stationary vacuum spacetimes admits Banach manifold structure. 
\end{abstract}
\maketitle

\section{Introduction}
A stationary spacetime $(V^{(4)},g^{(4)})$ is a 4-manifold $V^{(4)}$ equipped with a smooth Lorentzian metric $g^{(4)}$ of signature $(-,+,+,+)$, which in addition admits a time-like Killing vector field. A trivial example is the Minkowski space $(\mathbb{R}^{4}, g_{\text{Min}})$, where
	$$g_{\text{Min}}= -dt^2 + dx_1^2 + dx_2^2 + d x_3^2.$$
	
Stationary vacuum spacetimes are stationary spacetimes $(V^{(4)},g^{(4)})$ that solve the vacuum Einstein equation
\begin{equation}
Ric_{g^{(4)}}=0,
\end{equation}
where $Ric_{g^{(4)}}$ is the Ricci tensor of the metric $g^{(4)}$. Stationary vacuum spacetimes are important and much studied in general relativity. Two famous nontrivial examples (cf.[W]) are the Schwarzschild metric,
$$ds^2=-(1-\frac{2M}{r})dt^2+(1-\frac{2M}{r})^{-1}dr^2+r^2(d\theta^2+sin^2\theta d\phi^2);$$
and the Kerr metric,
\begin{equation*}
\begin{split}
ds^2=&-(1-\frac{2Mr}{\Sigma})dt^2-\frac{4Marsin^2\theta}{\Sigma}dtd\phi+\frac{\Sigma}{\Delta}dr^2+\Sigma d\theta^2\\
&+(r^2+a^2+\frac{2Ma^2rsin^2\theta}{\Sigma})sin^2\theta d\phi^2.
\end{split}
\end{equation*}

Throughout this paper, we assume that the spacetime $(V^{(4)},g^{(4)})$ is $globally$ $hyperbolic$, i.e. it admits a Cauchy surface $\Sigma$. The topology of a globally hyperbolic spacetime $V^{(4)}$ is necessarily $\Sigma\times \mathbb{R}$. In this case, we can define a global time function $t$ on $V^{(4)}$ so that every surface of constant $t$ is a Cauchy surface. When the spacetime $(V^{(4)},g^{(4)})$ is stationary in addition, we choose local coordinates $\{t,x^i\}~(i=1,2,3)$ so that $\partial_t$ is the time-like Killing field. Since $\partial_t$ generates a $\mathbb R-$ action of isometries on the spacetime, $g^{(4)}$ is determined by its value on the initial time slice, which will be denoted as $M=\{t=0\}$ in the following.  

We recall two well-known formulations of the stationary vacuum field equations --- the hypersurface formalism and the projection formalism.

In the hypersurface formalism, the stationary spacetime metric $g^{(4)}$ can be written globally in the form
\begin{equation}
g^{(4)}=-N^2dt^2+(g_M)_{ij}(dx^i+Y^idt)(dx^j+Y^jdt),
\end{equation}
Since $\partial_t$ is a Killing vector field, in the expression above the tensor fields $(g_M,N,Y)$ are all independent of $t$. So they can be regarded as tensor fields on $M$. The scalar field $N$ is called the lapse function and the vector field $Y=Y^i\partial_{x_i}$ is the shift vector. The symmetric tensor $g_M=(g_M)_{ij}dx^i dx^j$ is exactly the induced Riemannian metric of $M\subset(V^{(4)}, g^{(4)})$.

The stationary spacetime $(V^{(4)},g^{(4)})$ is vacuum if and only if the following stationary vacuum field equations hold for $(N,Y,g_M)$ on $M$, cf.[M], 
\begin{equation}
\begin{cases}
2NK-L_Yg=0,\\
Ric_{g_M}+(trK)K-2K^2-\frac{1}{N}D^2 N+\frac{1}{N}L_YK=0,\\
s_{g_M}+(trK)^2-|K|^2=0,\\
\delta K+d(trK)=0.
\end{cases}
\end{equation}
In the system above, $D^2N$ denotes the Hessian of function $N$ with respect to the metric $g_M$; $s_{g_M}$ denotes the scalar curvature of the metric $g_M$ on $M$; and $K$ is the second fundamental form of the hypersurface $M\subset (V^{(4)},g^{(4)})$.

It is known and easy to see that when the hypersurface $M$ is a closed $3-$manifold, there are no non-flat stationary vacuum solutions to the field equations (1.3). Lichnerowicz (cf.[L]) proved that a geodesically complete stationary vacuum spacetime is necessarily Minkowski spacetime $(\mathbb R^4, g_{\text{Min}})$ when the $3-$manifold $M$ is complete and asymptotically flat, cf. also [A2] for a generalization with no asymptotic condition.

Thus nontrivial solutions of $(1.3)$ only exist on 3-manifolds with nonempty boundary.  The standard case is when $M$ is diffeomorphic to $B^3$ (the interior case) or $\mathbb{R}^3\setminus B^3$ (the exterior case), where $B^3$ is the unit $3-$ball. This paper concerns the exterior case with asymptotical flatness conditions, cf.$\S$ 2.1.

Since the boundary $\partial M$ is nonempty, the issue of boundary conditions arises. In [B1], through a Hamiltonian analysis of the vacuum Einstein equation, Bartnik proposed a collection of boundary data given by 
$$(g_M^T, H_{g_M}, K(\mathbf n)^T, tr^TK),$$ 
which consists of the induced metric $g_M^T$ on the boundary $\partial M\subset(M,g_M)$, the mean curvature $H_{g_M}$ of $\partial M\subset (M,g_M)$, the mixed (perpendicular and tangential) part $K(\mathbf{n})^T$ of $K$ on $\partial M$, and the tangential trace $tr^TK$ of the second fundamental form $K$ along $\partial M$.

In [B1], Bartnik conjectured that among all the admissible asymptotically flat extensions of such boundary data to $\mathbb{R}^3\setminus B^3$, the infimum of the ADM mass is realized by one arising from a stationary vacuum spacetime -- that is a set of data $(g_M,N,Y)$ on $M$ satisfying the stationary vacuum Einstein field equations (1.3). A natural question raised in [B1] is whether the Bartnik boundary conditions are elliptic for the stationary vacuum Einstein field equations. This is one of the main motivations for the present paper. 

To check the ellipticity of boundary problems for the stationary vacuum equations, it is essential to find the right gauge terms and compute the principal symbols. However, it is very complicated to carry out this work with the system (1.3) in the hypersurface formalism. So alternatively, we will use the projection formalism.

In the projection formalism, we use $S$ to denote the collection of all trajectories of the Killing field $\partial_t$ in $(V^{(4)},g^{(4)})$, i.e. $S$ is the orbit space of the action of 1-parameter group $\mathbb{R}$ generated by $\partial_t$. Since the spacetime is globally hyperbolic, the qoutient space S is  a smooth 3-manifold and the metric $g^{(4)}$ restricted to the horizontal distribution --- the orthogonal complement of $\text{span}\{\partial_t\}$ in $TV^{(4)}$ --- induces a well-defined Riemannian metric $g_S$ on $S$. Let $\pi: V^{(4)}\rightarrow S$ denote the projection map, then metric $g^{(4)}$ is globally of the form
\begin{equation*}
g^{(4)}=-e^{2u}(dt+\pi^*\theta)^2+\pi^*g_S.
\end{equation*}
Here $\theta$ is a 1-form on $S$. It is shown in [G] that the pull back map $\pi^*$ gives a one-to-one correspondence between tensor fields $T^{'b...d}_{a...c}$ on $S$ and tensor fields $T^{b...d}_{a...c}$ on $V^{(4)}$ which satisfy
\begin{equation*}
(\partial_t)^aT^{b...d}_{a...c}=0,...,(\partial_t)_dT^{b...d}_{a...c}=0\text{ and }L_{\partial_t}T^{b...d}_{a...c}=0.
\end{equation*}
In the following we will omit the notation $\pi^*$ and identify tensor fields being on $S$ as tensor fields in $V^{(4)}$ satisfying the conditions above. So the spacetime metric can be written in a simpler form
\begin{equation}
g^{(4)}=-e^{2u}(dt+\theta)^2+g_S.
\end{equation}
The dual of the Killing vector field $\partial_t$ with respect to $g^{(4)}$ is given by $\xi=-e^{2u}(dt+\theta)$. 
The $twist$ $tensor$ $\omega$ is defined as 
\begin{equation}
\omega=\frac{1}{2}\star_{g^{(4)}} (\xi\wedge d\xi),
\end{equation}
where $\star_{g^{(4)}}$ denotes the hodge star operator of the metric $g^{(4)}$. The twist tensor provides a measurement of the integrability of the horizontal distribution $TS$ in $V^{(4)}$. It is actually a $1-$form on $S$ in the way discussed above, because equation $(1.5)$ is equivalent to 
$$\omega=-\frac{1}{2}e^{3u}\star_{g_S} d\theta.$$

If we perform a reparametrization of the time
$t'=t+f,$
where $f$ is a function on $S$, the formula $(1.4)$ becomes 
$$g^{(4)}=-e^{2u}(dt'+\theta')+ g_S,$$
with $\theta'=\theta -df$. It is easy to see the twist tensor $\omega$ remains invariant under this gauge transformation. Therefore, a stationary spacetime $(V^{(4)}, g^{(4)})$ corresponds uniquely to a collection of data $(g_S,u,d\theta)$ or $(g_S,u,\omega)$ on the quotient manifold $S$.  We refer to [K] and [CH] for more details of the projection formalism.

Notice that the restriction ($\pi|_{M}:M\rightarrow S$) of the projection $\pi$, gives a diffeomorphism between the hypersurface $M$ and the quotient manifold $S$. Thus boundary value problems in the setting $\{M,(g_M,N,Y)\}$ can be transferred to equivalent boundary value problems of $\{S,(g_S,u,\omega)\}$ via this diffeomorphism and vice versa. In certain respects, the projection formalism is more canonical, since there are many distinct hypersurfaces giving rise to the same stationary solution on the $4-$manifold, but the projection data is unique.

The stationary vacuum field equations in the projection formalism, which are equivalent to $(1.3)$ in the hypersurface formalism, are given by, cf.[H1],[H2],
\begin{equation*}
\begin{cases}
Ric_{g_S}-D^2u-(du)^2-2e^{-4u}(\omega\otimes\omega-|\omega|^2g_S)=0,\\
\Delta_{g_S} u-|du|^2-2e^{-4u}|\omega|^2=0,\\
\delta\omega+3 \langle du,\omega\rangle=0,\\
d\omega=0.
\end{cases}
\end{equation*}

Here $\Delta_{g_S}$ denotes the geometric Laplace operator of the metric $g_S$, i.e. $\Delta_{g_S} u=-tr_{g_S}D^2u$.
The last equation indicates that $\omega$ is exact. In the case $S\cong \mathbb R^3\setminus B^3$, we can assume $\omega=d\phi$ for some function $\phi$ on $S$. Thus the system above can be expressed equivalently as,
\begin{equation}
\begin{cases}
Ric_{g_S}-D^2u-(du)^2-2e^{-4u}(d\phi\otimes d\phi-|d\phi|^2g_S)=0,\\
\Delta_{g_S} u-|du|^2-2e^{-4u}|d\phi|^2=0,\\
\Delta_{g_S}\phi+3 \langle du,d\phi\rangle=0.
\end{cases}
\end{equation}
With the system above, the work of choosing proper gauge terms and dealing with the principal symbols turns out to be much easier, compared with the system $(1.3)$. Thus it is of interest to study the ellipticity of the system $(1.6)$ with geometrically natural prescribed boundary conditions on the quotient manifold $S$. 

Rather than transforming the Bartnik boundary conditions from the slice $M$ to $S$, here we analyze some simpler boundary conditions arising naturally from the projection formalism. In view of the Bartnik conditions, we first choose to prescribe $(g^T_S,H_{g_S})$ --- the induced metric $g_S^T$ on the boundary $\partial S\subset(S,g_S)$ and  the mean curvature $H_{g_S}$ of $\partial S\subset (S,g_S)$. In addition, we pose a restriction on the twist tensor $\omega$, by prescribing $\mathbf n_{g_S}(\phi)$ on $\partial S$, where $\mathbf{n}_{g_S}$ is the unit normal vector of the boundary pointing outwards.  Actually the collection of boundary data,
\begin{equation}
\{g_S^T,H_{g_S}, \mathbf{n}_{g_S}(\phi)\},
\end{equation}
also arises naturally from the boundary terms in the variation of a functional on $S$, which comes from the reduction of the Einstein Hilbert action from $V^{(4)}$ to $S$, c.f.\S 3.

The first main theorem we will prove is the ellipticity of the boundary data (1.7).
\begin{theorem}
The stationary vacuum field equations $(1.6)$ and boundary conditions $(1.7)$ form an elliptic boundary value problem, modulo gauge transformations.
\end{theorem}
To prove this theorem, we first present in \S 2 the conformal transformation of the vacuum field equations, which gives a differential operator with simpler symbols. For the purpose of ellipticity, we modify the equations using certain gauge terms. After that, in \S3 ellipticity of the boundary conditions (1.7) is proved with respect to different choices of gauge terms.

\medskip
$Remark.$ The method we use to prove ellipticity in this paper can be applied to more general boundary value problems for the stationary vacuum field equations. 
\medskip

In \S4 we prove a manifold structure theorem for the moduli space $\mathcal E_C=\mathcal E_C^{m,\alpha}$ of stationary vacuum spacetimes. The space $\mathcal E_C$ is basically the space of all $C^{m,\alpha}$ asymptotically flat stationary vacuum solutions to the system $(1.6)$ on $S$ modulo the action of the group $\mathcal D_0^{m+1,\alpha}(S)$ of diffeomorphisms on $S$ that are equal to the identity map when restricted on $\partial S$. In addition, based on the boundary conditions $(1.7)$, we have a natural map $\Pi$, from the moduli space $\mathcal E_C$ to the space of boundary data defined as follows,
\begin{equation}
\begin{split}
\Pi: &\mathcal E_C\rightarrow Met^{m,\alpha}(\partial S)\times C^{m-1,\alpha}(\partial S)\times C^{m-1,\alpha}(\partial S),\\
&\Pi[(g_S,u,\phi)]=(g_S^T,H_{g_S},\mathbf n_{g_S}(\phi)).
\end{split}
\end{equation}
Here $Met^{m,\alpha}(\partial S)$ is the space of $C^{m,\alpha}$ metrics on $\partial S$; $C^{m-1,\alpha}(\partial S)$ is the space of $C^{m-1,\alpha}$ functions on $\partial S$.
By applying the ellipticity result, we will prove the following theorem.
\begin{theorem}
The moduli space $\mathcal E_C$ is an infinite dimensional $C^{\infty}$ Banach manifold, and the map $\Pi$ is $C^{\infty}$ smooth and Fredholm, of Fredholm index 0.
\end{theorem}

The theorems we prove in this paper are generalizations of the results proved in [AK], where spacetimes are static. Related work can be found in [J],[M1-3],[R] and elsewhere. In a sequel to this work, we plan to discuss ellipticity of the more complicated Bartnik boundary conditions. \\

\textbf{Acknowledgements}: I would like to express great thanks to my advisor Michael Anderson for suggesting this problem and for valuable discussions and comments.

\section{Background Discussion}
\subsection{Asymptotic flatness}$~$

Throughout this paper, the 3-manifold $S$ is assumed to be diffeomorphic to $\mathbb{R}^3\setminus B^3$, with $B^3$ the unit $3-$ball. When it goes to the infinite end of $S$, we assume that the data $(g_S,u,\phi)$ is asymptotically flat, in the sense that
$$ g-g_{F}\rightarrow 0,~u\rightarrow 0,~\phi\rightarrow 0,~~~~ as~r\rightarrow \infty$$
where $g_{F}$ is the flat metric on $\mathbb{R}^3\setminus B^3$ and $r$ is the radius function on $S$ obtained by pulling back the radius function on $\mathbb{R}^3\setminus B^3$ under a fixed diffeomorphism. To describe rigorously the decay behavior above, we use the weighted H\"older spaces defined as follows, cf. [B2], [LP].
\begin{definition} 
On the manifold $S\cong\mathbb R^3\setminus B^3$ which is equipped with local coordinates $\{x^i\}~(i=1,2,3)$ and a radius function $r$, we define several Banach spaces for $m\in\mathbb{N}$, and $\alpha,\delta\in\mathbb{R}$:
\begin{equation*}
    \begin{split}
&C^m_{\delta}(S)=\{\text{functions}~v~\text{on}~S: ||v||_{C^m_{\delta}}=\Sigma_{k=0}^msup~r^{k+\delta}|\nabla^kv|<\infty\},\\
&C^{m,\alpha}_{\delta}(S)=\{\text{functions}~v~\text{on}~S:\\ &\quad\quad\quad\quad\quad||v||_{C^m_{\delta}}+sup_{x,y}[\text{min}(r(x),r(y))^{m+\alpha+\delta}
\frac{\nabla^mv(x)-\nabla^mv(y)}{|x-y|^{\alpha}}]
<\infty\},\\
&Met^{m,\alpha}_{\delta}(S)=\{\text{metrics}~g~\text{on}~S: (g_{ij}-\delta_{ij})\in C^{m,\alpha}_{\delta}(S)\},\\
&T^{m,\alpha}_{\delta}(S)=\{\text{vector fields}~X~\text{on}~S: X^i\in C^{m,\alpha}_{\delta}(S)\},\\
&(T_p)^{m,\alpha}_{\delta}(S)=\{p-\text{tensors}~\tau~\text{on}~S: \tau_{i_1i_2..i_p}\in C^{m,\alpha}_{\delta}(S)\},\\
&(\wedge_p)^{m,\alpha}_{\delta}(S)=\{p-\text{forms}~\sigma~\text{on}~S: \sigma_{i_1i_2..i_p}\in C^{m,\alpha}_{\delta}(S)\},\\
&(S_2)^{m,\alpha}_{\delta}(S)=\{\text{symmetric }2-\text{forms}~h~\text{on}~S: h_{ij}\in C^{m,\alpha}_{\delta}(S)\}.
    \end{split}
\end{equation*}

\end{definition}
\begin{definition}
The data $(g_S,u,\phi)$ is called \textit{asymptotically flat of order} $\delta$ if 
\begin{equation}
(g_S,u,\phi)\in [Met^{m,\alpha}_{\delta}\times C^{m,\alpha}_{\delta}\times C^{m,\alpha}_{\delta}](S),
\end{equation}
for some $m,~\alpha$ and $\delta$.
\end{definition}
Throughout the following, the orders $m,\alpha$ and the decay rate $\delta$ are fixed, and chosen to satisfy,
 $$m\geq 2,~0<\alpha<1,~\frac{1}{2}<\delta<1.$$

$Remark.$ In the previous section, we introduced the diffeomorphism  $\pi|_{M}:M\rightarrow S$ between the hypersurface $M$ and the quotient space $S$. In fact, under this diffeomorphism, the asymptotic flatness condition (2.1) of
 $(g_S,u,\phi)$ in $S$ is equivalent to the asymptotic condition in $M$:
 $$(g_M,N,Y)\in [Met^{m,\alpha}_{\delta}\times C^{m,\alpha}_{\delta}\times (\wedge_1)^{m,\alpha}_{\delta}](M),$$
which, furthermore, is equivalent to the decay behavior as in Bartnik's work.

\subsection{Conformal transformation}$~~$

To simplify the symbols of the stationary field equations $(1.6)$, we first apply a conformal transformation
\begin{equation}
    g=e^{2u}g_S. 
\end{equation}
Under such a transformation, the data $(g_S,u,\phi)$ is in 1-1 correspondence to the triple $(g,u,\phi)$; and if $(g_S,u,\phi)$ is asymptocially flat as defined in (2.1), then so is the data $(g,u,\phi)$, i.e.
$$(g,u,\phi)\in [Met^{m,\alpha}_{\delta}\times C^{m,\alpha}_{\delta}\times C^{m,\alpha}_{\delta}](S).$$ 

Furthermore, the stationary vacuum field equations (1.6), which are expressed in terms of $(g_S,u,\phi)$, can be simplified equivalently to the following system with unknowns $(g,u,\phi)$, cf. [K],
$$\text{(I)}~~\begin{cases}
Ric_g-2du\otimes du-2e^{-4u}d\phi\otimes d\phi=0,\\
\Delta_g u-2e^{-4u}|d\phi|^2=0,\\
\Delta_g \phi+4\langle du,d\phi\rangle=0.
\end{cases}$$
Here the norm $|d\phi|$ and the inner product $\langle du,d\phi\rangle$ are all with respect to the metric $g$. We point it out that the field equations above can be expressed in another equivalent way, where the Ricci tensor in $Ric_g$ is replaced by the Einstein tensor $Ein_g=Ric_g-\frac{1}{2}s_gg$. In fact, the trace of the first equation in (I) is given by
$$s_{g}-2|du|^2-2e^{-4u}|d\phi|^2.$$
Let $T_g$ be as, $$T_g=\frac{1}{2}(s_{g}-2|du|^2-2e^{-4u}|d\phi|^2)g.$$
Then it is easy to see that, system (I) is equivalent to the following system (II) by inserting $T_g$ into the first equation,
$$\text{(II)}\begin{cases}
Ric_g-2du\otimes du-2e^{-4u}d\phi\otimes d\phi-T_g=0,\\
\Delta_g u-2e^{-4u}|d\phi|^2=0,\\
\Delta_g \phi+4\langle du,d\phi\rangle=0.
\end{cases}
$$
Rearranging the order of terms, we can rewrite the first equation in (II) as
\begin{equation}
(Ric_g-\frac{1}{2}s_gg)-2du\otimes du-2e^{-4u}d\phi\otimes d\phi+(|du|^2+e^{-4u}|d\phi|^2)g=0,
\end{equation}
where the leading term --- the term with highest order of derivative of $g$ --- is exactly the Einstein tensor $(Ric_g-\frac{1}{2}s_gg)$.

Observe that the system (I) (or (II)) is not elliptic, because the full system is invariant under diffeomorphisms, i.e. if $(g,u,\omega)$ solves the vacuum Einstein field equations, then the pull back data $(\Psi^*g,\Psi^*u,\Psi^*\omega)$ under some diffeomorphism $\Psi$ on $S$, is also a solution. So to ensure ellipticity, as is usual we modify the system using gauge terms, and obtain
$$\text{(III)}\begin{cases}
Ric_g-2du\otimes du-2e^{-4u}d\phi\otimes d\phi+T+\delta^{\ast}G=0,\\
\Delta_g u -2e^{-4u}|d\phi|^2=0,\\
\Delta_g \phi+4\langle du,d\phi\rangle=0.
\end{cases}
$$

We can set the pair $(T,G)$ to be  
$$
\begin{cases}
T_1=0,\\
G_1=\beta_{\tilde{g}}(g),
\end{cases}
$$
 where $\tilde{g}$ is a reference metric near $g$ in $Met^{m,\alpha}_{\delta}(S)$, and $\beta$ is the Bianchi operator defined as,
 $\beta_{\tilde g}g=\delta_{\tilde g}g+\frac{1}{2}dtr_{\tilde g}g$, with $\delta=-tr\nabla$ being the divergence operator. 
 This choice is to insert $G_1=\beta_{\tilde{g}}(g)$ (the Bianchi gauge) into the system (II).
 
 An alternative way is to set $(T,G)$ as follows,
$$\begin{cases}
T_2=-T_g,\\
G_2=\delta_{\tilde{g}}(g).
\end{cases}$$
This corresponds to inserting $G_2=\delta_{\tilde{g}}(g)$ (the divergence gauge) into (II).

We will be concerned with both the distinct choices of gauge terms here. In the case $(T,G)=(T_1,G_1)$, the principal symbols of the system (III) are simple and ellipticity can be proved by straightforward computation. However, such a system is not self-adjoint, which makes it not suitable for the proof of the manifold theorem in \S4. On the other hand, the system (III) with  $(T,G)=(T_2,G_2)$ is formally self-adjoint, whereas its principal symbols are much more complicated. We will use the ellipticity result of the case $(T,G)=(T_1,G_1)$ to prove the ellipticity for the gauge choice $(T_2,G_2)$. We refer to \S3 for more details.

Since the boundary $\partial S$ is not empty, it is necessary to include a boundary condition for the gauge term $G$. A convinient choice is
\begin{equation}
G=0~~\quad\quad~~\text{on}~\partial S.
\end{equation}
In the following we will prove that, equipped with this boundary restriction, solutions to the gauged system (III) with $(T,G)=(T_2,G_2)$, correspond to solutions to the stationary vacuum system (II) modulo diffeomorphisms.
\subsection{Moduli space of stationary vacuum spacetimes}$~~$

We begin by defining the following spaces of stationary vacuum spacetimes.
\begin{definition}
\begin{equation*}
\begin{split}
\mathbb{E}_C&=\{(g,u,\phi)\in [Met_{\delta}^{m,\alpha}\times C_{\delta}^{m,\alpha}\times C_{\delta}^{m,\alpha}](S):~(g,u,\phi)\text{ solves (II)}\},\\ 
\mathbf{Z}_C&=\{(g,u,\phi)\in [Met_{\delta}^{m,\alpha}\times C_{\delta}^{m,\alpha}\times C_{\delta}^{m,\alpha}](S):~\\
&(g,u,\phi)\text{ solves (III) where}~(T,G)=(T_2,G_2)\text{ and }G_2=0\text{ on the boundary }\partial S\}.
\end{split}
\end{equation*}
\end{definition}
The following lemma shows that $\mathbf{Z}_C\subset \mathbb{E}_C$.
\begin{lemma}
Elements in $\mathbf{Z}_C$ are elements in $\mathbb{E}_C$ which satisfy the divergence gauge condition $\delta_{\tilde g}g=0$ in $S$.
\end{lemma}
\begin{proof}
It suffices to prove the gauge term $G_2=0$ globally for $(g,u,\phi)\in \mathbf Z_C$. From the equation $(2.3)$ we know that, if $(T,G)=(T_2,G_2)$ then leading term of the first equation in (III) is the Einstein tensor $Ein_g$, where we have the Bianchi identity, $\delta_gEin_g=0$. Thus, taking the divergence (with respect to $g$) on it, we obtain
$$\delta_g\{Ein_g-2du\otimes du-2e^{-4u}d\phi\otimes d\phi+(|du|^2+e^{-4u}|d\phi|^2)g+\delta^{\ast}G_2 \}=0,$$ 
which gives,
$$\delta_g\{-2du\otimes du-2e^{-4u}d\phi\otimes d\phi+(|du|^2+e^{-4u}|d\phi|^2)g\}+\delta\delta^{\ast}G_2=0.$$
Basic computation gives
\begin{equation*}
\begin{split}
&\delta\{-2du\otimes du-2e^{-4u}d\phi\otimes d\phi+(|du|^2+e^{-4u}|d\phi|^2)g\}\\
&=-2(\Delta_g u)du-8e^{-4u}\langle du,d\phi\rangle d\phi-2e^{-4u}(\Delta_g\phi)d\phi+4e^{-4u}du|d\phi|^2.
\end{split}
\end{equation*}
Combining with the second and third equations in system (III), it is easy to derive that the expression above is equal to zero, and consequently 
$$\delta\delta^{\ast}G_2=0.$$
Thus we obtain the following system for $G_2$,
$$
\begin{cases}
\delta\delta^{\ast}G_2=0~~~~~\text{on}~~S,\\
G_2=0~~~~~~~~\text{on}~~\partial S.
\end{cases}$$
Integration by parts gives: 
$$0=\int_S \langle \delta\delta^{\ast}G_2,G_2\rangle=\int_S|\delta^{\ast}G_2|^2-\int_{\partial S} \delta^{\ast}G_2(\mathbf n,G_2)-\int_{\partial S_{\infty}} \delta^{\ast}G_2(\mathbf n,G_2).$$
Here the boundary integral on $\partial S$ must vanish because $G_2=0$ on the boundary. The boundary term at infinity $\int_{\infty}=\lim_{r\to\infty}\int_{S_r}$, with $S_r$ denoting the sphere of radius $r$ on $S$. It is also zero because the term $\delta^*G_2(\mathbf n,G_2)$ decays at the rate $r^{-2\delta-2}<r^{-3}$. Then it follows from the equation above that $\delta^{\ast}G_2=0$ in $S$, i.e. $G_2$ is a Killing vector field. In addition, $G_2$ is vanishing on $\partial S$. Thus, we have $G_2=0$. This completes the proof.
\end{proof}
$Remark.$ The lemma above shows that adding the divergence gauge to the system (II) preserves the stationary vacuum property of the solutions. In contrast, it is unknown in general whether adding the Bianchi gauge to system (I) will work in the same way. In fact, when $(T,G)=(T_1,G_1)$, the leading term in the first equation of (III) is the Ricci tensor $Ric_g$. Thus instead of taking divergence as in the proof above, one needs to apply the Bianchi operator to that equation, which yields $\beta\delta^{*}G_1=0$. The operator $\beta\delta^*$ is not positive in general, so the argument above does not apply to the Bianchi gauge. 

\medskip
Next, we will show  $\mathbb{E}_C\subset\mathbf{Z}_C$ in the sense of moduli space.

First define a Banach space $\mathcal X^{m,\alpha}_{\delta}(S)$ of asymptotically flat vector fields vanishing on $\partial S$:
$$\mathcal X^{m,\alpha}_{\delta}(S)=\{X\in T^{m,\alpha}_{\delta}(S):~X=0~\text{on }\partial S\}.$$
 Then the following lemma holds for the space $\mathcal X^{m,\alpha}_{\delta}(S)$. 
\begin{lemma}
 The map $\delta\delta^{\ast}: \mathcal X^{m,\alpha}_{\delta}(S)\rightarrow T^{m-2,\alpha}_{\delta+2}(S)$ is an isomorphism.
\end{lemma}
\begin{proof}
From the proof of the previous lemma, one sees that kernel of  $\delta\delta^{\ast}: \mathcal X^{m,\alpha}_{\delta}(S)\rightarrow T^{m-2,\alpha}_{\delta+2}(S)$ is zero. On the other hand, $\delta\delta^{\ast}$ is an elliptic operator with Fredholm index 0, thus it is an isomorphism.
\end{proof}
Next, let $\mathcal D_0^{m+1,\alpha}(S)$ be the group of $C^{m+1,\alpha}_{\delta}$ diffeomorphisms of $S$ which equal to the identity map on $\partial S$. These are diffeomorphisms decaying asymptotically to the identity at the rate $r^{-\delta}$. The group $\mathcal D_0^{m+1,\alpha}(S)$ acts freely and continuously on $Met_{\delta}^{m,\alpha}(S)$ by pull back and one has the following local result.
\begin{theorem}
Given any $g\in Met^{m,\alpha}_{\delta}(S)$ near $\tilde g$, there is a unique diffeomorphism $\Psi\in\mathcal D_0^{m+1,\alpha}(S)$ near the identity map $Id_S$ on $S$ such that the pull back metric $\Psi^{\ast}g$ satisfies the divergence gauge condition $\delta_{\tilde g}(\Psi^{\ast}g)=0$.
\end{theorem}
\begin{proof}: Define a map $\mathcal F$ as follows, $$\mathcal F: [\mathcal D_0^{m+1,\alpha}\times   Met^{m,\alpha}_{\delta}](S) \rightarrow (\Lambda_1)^{m-1,\alpha}_{\delta+1}(S), $$
$$\mathcal F(\Psi,g)=\delta_{\tilde g}(\Psi^{\ast}(g)).$$
Linearization of $\mathcal F$ at $(Id_S, \tilde g)$ with respect to a deformation $(X,h)\in T[\mathcal D_0^{m+1,\alpha}\times   Met^{m,\alpha}_{\delta}](S)$ is given by
$$D_0\mathcal F(X,h)=\delta\delta^{\ast}X+\delta(h).$$ 
Here $X\in\mathcal X^{m+1,\alpha}_\delta(S)$, and hence the first part $\delta\delta^{\ast}$ in the linearization above is an isomorphism by the previous lemma. According to the inverse function theorem, for any $g$ in a neigbourhood of $\tilde g$, there exists a unique $\Psi$ near $Id_S$ such that $F(\Psi,g)=0$, which proves the theorem.
\end{proof}
Now define the moduli space $\mathcal{E}_C=\mathcal{E}_C^{m,\alpha}$ to be the quotient of the space $\mathbb E_C$ by the diffeomorphism group:
$$\mathcal{E}_C=\mathbb{E}_C/\mathcal D_0^{m+1,\alpha}(S).$$ 
By Lemma 2.4, any element of $\mathbf{Z}_C$ belongs to one of the equivalence classes in $\mathcal E_C$. Conversely, given any stationary vacuum data $(g,u,\phi)$ near $\tilde g$, according to the theorem above, one can choose a unique diffeomorphism $\Psi\in\mathcal D_0^{m+1,\alpha}(S)$ near $Id_S$ so that $\delta_{\tilde g}(\Psi^*g)=0$, i.e. the pull back data $(\Psi^*g,\Psi^*u,\Psi^*\phi)$ belongs to $\mathbf{Z}_C$. Therefore, locally elements in the set $\mathbf{Z}_C$ near $\tilde g$ are in 1-1 correspondence to equivalence classes in the moduli space $\mathcal{E}_C$ near $[\tilde g]$. In other words, $\mathbf Z_C$ can be taken as a local coordinate chart for $\mathcal E_C$.

\subsection{Boundary conditions}$~~$

As in the introduction, we pose a geometrically natural collection of boundary conditions on $\partial S$:
\begin{equation}
\begin{cases}
g_S^T=\gamma,\\
H_{g_S}=\lambda,\\
\mathbf n_{g_S}(\phi)=f,
\end{cases}
\end{equation}
where $\gamma\in Met^{m,\alpha}(\partial S)$ is a fixed metric of the surface $\partial S$; and $\lambda, f\in C^{m-1,\alpha}(\partial S)$ are prescribed functions on $\partial S$. 
Under the conformal transformation $(2.2)$, these tensor fields become
\begin{equation*}
g_S^T=e^{-2u}g^T,~H_{g_S}=e^{u}(H_g-2\mathbf n_g(u)),~\mathbf n_{g_S}=e^{u}\mathbf n_g.
\end{equation*}
Thus one can translate the boundary conditions (2.5) to the following boundary conditions for the data $(g,u,\phi)$,
\begin{equation}
\begin{split}
\begin{cases}
e^{-2u}g^T=\gamma\\
H_g-2\mathbf n_g(u)=e^{-u}\lambda\\
\mathbf n_g(\phi)=e^{-u}f
\end{cases}
\text{on }\partial S.
\end{split}
\end{equation}
Pairing these boundary conditions with the gauged field equations (III), we obtain a boundary value problem,
\begin{equation}
\begin{split}
&\begin{cases}
Ric_g-2du\otimes du-2e^{-4u}d\phi\otimes d\phi+T+\delta^{\ast}G=0\\
\Delta_g u -2e^{-4u}|d\phi|^2=0\\
\Delta_g \phi+4<du,d\phi>=0
\end{cases}\quad\text{on $S$,}\\
&\begin{cases}
G=0\\
e^{-2u}g^T=\gamma\\
H-2\mathbf n(u)=e^{-u}\lambda\\
\mathbf n(\phi)=e^{-u}f
\end{cases}\quad\text{on $\partial S$.}
\end{split}
\end{equation}
From now on, we will omit the subscript $g$ of $H_g,\mathbf n_g$ when there is no room for confusion. 
The main step to prove Theorem 1.1 is verifying that the boundary value problem above is elliptic.
To do this, we define a differential operator $\mathcal{P}=(\mathcal{L},\mathcal{B})$ based on it, where $\mathcal{L}$ denotes the interior operator and $\mathcal{B}$ the boundary operator. The interior operator $\mathcal L$, mapping the data $(g,u,\phi)$ to the interior equations of $(2.7)$, is defined as follows,
\begin{equation*}
\begin{split}
\mathcal{L}:[Met_{\delta}^{m,\alpha}\times &C^{m,\alpha}_{\delta}\times C^{m,\alpha}_{\delta} ](S)
\rightarrow [(S_2)^{m-2,\alpha}_{\delta+2}\times C^{m-2,\alpha}_{\delta+2}\times C^{m-2,\alpha}_{\delta+2}](S),\\
\mathcal{L}(g,u,\phi)=\{~&2(Ric_g-2du\otimes du-2e^{-4u}d\phi\otimes d\phi+\delta^{\ast}G+T),\\
& 8(\Delta_g u -2e^{-4u}|d\phi|^2),\\
& 8e^{-4u}(\Delta_g \phi +4\langle du,d\phi\rangle)~\}.
\end{split}
\end{equation*}
We refer to Definition 2.1 for the notations of different spaces of tensor fields involved above. Notice that we add extra scalar factors $2$, $8$ and the function $8e^{-4u}$ in the front of the equations. They do not affect ellipticity of the boundary value problem but they are necessary in proving self-adjointness in the following sections.

The boundary operator $\mathcal B$, mapping the data $(g,u,\phi)$ to the boundary equations in $(2.7)$, is given by,
\begin{equation*}
\begin{split}
\mathcal{B}:~[Met_{\delta}^{m,\alpha}\times C^{m,\alpha}_{\delta}\times &C^{m,\alpha}_{\delta}](S)\rightarrow \mathbf B^{m,\alpha}(S)\\
\mathcal{B}(g,u,\phi)&=
\{~G,\\
&\quad\quad e^{-2u}g^T-\gamma,\\
&\quad\quad H-2\mathbf n(u)-e^{-u}\lambda,\\
&\quad\quad\mathbf n(\phi)-e^{-u}f~\},
\end{split}
\end{equation*}
where we will use $\mathbf B^{m,\alpha}(S)$ to denote the target space of $\mathcal B$. Examining the boundary terms, one can see that $\mathbf B^{m,\alpha}(S)= [(T_1^{m-1,\alpha}\times C^{m-1,\alpha})\times S_2^{m,\alpha}\times C^{m-1,\alpha}\times C^{m-1,\alpha}](\partial S)$. Here the gauge $G$ is understood to be in the space $[T_1^{m-1,\alpha}\times C^{m-1,\alpha}](\partial S)$, because it is a $1$-form on $S$ which, on the boundary $\partial S$, consists of a $1$-form $G^T$ tangential to $\partial S$ and a normal component $G(\mathbf n)\cdot \mathbf n$ that yields a $C^{m-1,\alpha}$ scalar fields $G(\mathbf n)$ on $S$.

Throughout this paper, $\mathcal P$ will be written as $\mathcal P_1=(\mathcal L_1,\mathcal B_1)$ if the gauge terms in $\mathcal L,\mathcal B$ correspond to the Bianchi gauge, and $\mathcal P_2=(\mathcal L_2,\mathcal B_2)$ if the divergence gauge is applied.

Let $(g,u,\phi)$ be a fixed element in the zero set $\mathcal{P}^{-1}(0)$, and choose the reference metric $\tilde{g}=g$ in the gauge term $G$. The linearization of $\mathcal{P}$ at $(g,u,\phi)$ is given by
$$D\mathcal{P}(h,v,\sigma)=(D\mathcal{L}(h,v,\sigma),~D\mathcal{B}(h,v,\sigma)),$$
where $(h,v,\sigma)$ is an infinitesimal deformation of the data $(g,u,\phi)$. The operators $D\mathcal{L}$, $D\mathcal{B}$ are the linearizations of $\mathcal{L}$ and $\mathcal{B}$, given by,
\begin{equation*}
\begin{split}
D\mathcal{L}:~[(S_2)_{\delta}^{m,\alpha}\times C^{m,\alpha}_{\delta}\times C^{m,\alpha}_{\delta}](S)&
\rightarrow [(S_2)^{m-2,\alpha}_{\delta+2}\times C^{m-2,\alpha}_{\delta+2}\times C^{m-2,\alpha}_{\delta+2}](S),\\
D\mathcal{L}(h,v,\sigma)=\{~&D^{\ast}Dh-Z(h)+O_1,\\
&~~8\Delta_g v+O_1,\\
&~~~~8e^{-4u}(\Delta_g \sigma +4\langle du,d\sigma\rangle) +O_0~\},
\end{split}
\end{equation*}
and
\begin{equation*}
\begin{split}
D\mathcal{B}:~[(S_2)_{\delta}^{m,\alpha}\times C^{m,\alpha}_{\delta}\times &C^{m,\alpha}_{\delta}](S)\rightarrow\mathbf B^{m,\alpha}(S),\\
D\mathcal{B}(h,v,\sigma)&=\{~G'_h,\\
&\quad\quad e^{-2u}(-2vg+h)|_{\partial S},\\
&\quad\quad H'_h-2\mathbf n(v)+O_0,\\
&\quad\quad \mathbf n(\sigma)+O_0~\}.
 \end{split}
 \end{equation*}
In the expression above, $D^*D$ is the Laplace operator defined as $D^*Dh=-\nabla^i\nabla_ih$. The terms $Z$ and $G'_h$ depend on the choice of gauge terms. When the Bianchi gauge is choosen, they are of the form
\begin{equation}
\begin{cases}
Z_1(h)=0,\\
(G_1)'_h=\beta_{g}h.
\end{cases}
\end{equation}
If the divergence gauge is used,
\begin{equation}
\begin{cases}
Z_2(h)=D^2(trh)+\Delta_g(trh)g+(\delta\delta h) g,\\
(G_2)'_h=\delta_{g}h.
\end{cases}
\end{equation}

The expressions $O_1$ and $O_0$ stand for lower order terms which involve the derivative of $(h,v,\sigma)$ with order not higher than $1$ and $0$. In the linearization $D\mathcal B$, the term $H'_h$ denotes the variation of the mean curvature with respect to the deformation $h$. We refer to the appendix \S5 for the detailed calculation.

Since ellipticity only depends on the principal part of the operator, we can remove the lower order terms $O_1$ and $O_0$ in  $D\mathcal{P}$ and study the simplified operator $P(h,v,\sigma)=(L(h,v,\sigma), B(h,v,\sigma))$, where $L$ and $B$ are as follows:
\begin{equation}
\begin{split}
L:~[(S_2)_{\delta}^{m,\alpha}\times C^{m,\alpha}_{\delta}\times C^{m,\alpha}_{\delta}](S)
\rightarrow &[(S_2)^{m-2,\alpha}_{\delta+2}\times C^{m-2,\alpha}_{\delta+2}\times C^{m-2,\alpha}_{\delta+2}](S),\\
L(h,v,\sigma)=\{&~D^{\ast}Dh-Z(h),\\
&8\Delta_g v,\\
&8e^{-4u}(\Delta_g \sigma+4\langle du,d\sigma\rangle)~\},
\end{split}
\end{equation}
and
\begin{equation}
\begin{split}
B:~[(S_2)_{\delta}^{m,\alpha}\times C^{m,\alpha}_{\delta}\times &C^{m,\alpha}_{\delta}](S)
\rightarrow\mathbf B^{m,\alpha}(S),\\
B(h,v,\sigma)=\{~&G'_h,\\
&e^{-2u}(-2vg+h)|_{\partial S},\\
&H'_h-2\mathbf n(v),\\
&\mathbf n(\sigma)~\}.
\end{split}
\end{equation}
In the last component of $L$, we keep the lower order term $4\langle du,d\sigma\rangle$ for the purpose of self-adjointness which will be discussed later; again this does not affect the ellipticity. 

In the following section, if the pair $(Z,G'_h)$ takes the values in $(2.8)$, the operator $P$ will be denoted by $P_1=(L_1,B_1)$; and if equipped with the divergence gauge $(2.9)$, $P$ will be written as $P_2=(L_2,B_2)$. We will prove the ellipticity of both operators $P_1$ and $P_2$. As a consequence, the boundary value problem $(2.7)$ is elliptic with respect to both choices of gauges.
\section{ellipticity}
In this section, we will prove the ellipticity  for the operators $P_1,P_2$ defined at the end of last section, implementing the criterion developed by Agmon-Douglis-Nirenberg (cf.[ADN]). We use the following standard notation. Let $\xi$ denote a $1-$form on $S$; $\eta$ denote a nonzero $1-$form tangential to the boundary $\partial S$, i.e. $\eta(\mathbf n)=0$; and $\mu$ a unit $1-$form normal to the boundary $\partial S$, i.e. $\mu^T=0$. The index $0$ denotes the normal direction to $\partial S$, and index $1,2$ denote the tangential direction on $\partial S$. 

Let $P=(L,B)$ be a differential operator consisting of an interior operator $L$ and a boundary operator $B$. Denote the matrix of principal symbols of the interior operator at $\xi$ as $L(\xi)$ and the matrix of principal symbols of the boundary operator as $B(\xi)$. By [ADN], the operator $P$ forms an elliptic boundary value problem if and only if the following two conditions hold.
\\(A) (properly elliptic condition): determinant $l(\xi)$ of $L(\xi)$ has no nontrivial real root;
\\(B) (complementing boundary {condition}): Take the adjoint matrix $L^{\ast}(\xi)$ of $L(\xi)$. Let $\xi=(\eta+z\mu)$. {The} rows of $B(\eta+z\mu)\cdot L^{\ast}(\eta+z\mu)$ are linearly independent modulo $l^+(z)$, where $l^+(z)=\prod(z-z_k)$ and $\{z_k\}$ are the roots of $l(\eta+z\mu)=0$ having positive imaginary parts.
 
\subsection{Ellipticity with the Bianchi gauge}
\begin{theorem}
$P_1$ is an elliptic operator.
\end{theorem}
\begin{proof}
Since the operator $L_1$ is essentially the Laplace operator, it is easy to see that the matrix of principal symbols for $L_1$ at $\xi$ is given by
$$L_1(\xi)=
\begin{bmatrix}
|\xi|^2I_{6\times6}&0&0\\
0&8|\xi|^2&0\\
0&0&8e^{-4u}|\xi|^2
\end{bmatrix}.
$$
Then the adjoint matrix of $L(\xi)$ is 
$$L_1^{\ast}(\xi)=|\xi|^{14}\begin{bmatrix}64e^{-4u}I_{6\times6}&0&0\\0&8e^{-4u}&0\\0&0&8\end{bmatrix}.$$
The determinant of $L_1(\xi)$ is $l(\xi)=64e^{-4u}|\xi|^{16}$. So it is obvious that the interior operator is properly elliptic.
\\The root of $l(\eta+z\mu)$ with positive imaginary part is $z=i|\eta|$, and this implies $l^+(z)=(z-i|\eta|)^8$ in the criterion (B) above.
Let $C$ be a general vector in $\mathbb{C}^8$.  The complementing boundary condition holds if the equation below has no nontrivial solution in $\mathbb{C}^8$:
\begin{equation}
C\cdot B_1(\eta+z\mu)\cdot L_1^{\ast}(\eta+z\mu)=0\quad(\text{mod}~l^+(z) ).
\end{equation}
Since $L_1^*(\eta+z\mu)=(z+|\mu|^2)^{7}\begin{bmatrix}64e^{-4u}I_{6\times6}&0&0\\0&8e^{-4u}&0\\0&0&8\end{bmatrix}$,
one sees easily that equation $(3.1)$ is equivalent to the following,
$$C\cdot B_1(\eta+z\mu)\cdot\begin{bmatrix}64e^{-4u}I_{6\times6}&0&0\\0&8e^{-4u}&0\\0&0&8\end{bmatrix}=0\quad(\text{mod} (z-i|\eta|)).$$
And furthermore, this holds if and only if the following is true,
$$C\cdot B_1(\eta+i|\eta|\mu)\cdot\begin{bmatrix}64e^{-4u}I_{6\times6}&0&0\\0&8e^{-4u}&0\\0&0&8\end{bmatrix}=0.$$
So the equation (3.1) will have no non-zero solution, if the matrix of principal symbol $B_1(\eta+i|\eta|\mu)$ is non-degenerate. Write the nonzero tangential 1-form $\eta=(\eta_1,\eta_2)$. Basic computation (cf.\S5.2) shows that the principal symbols of the boundary operator $B_1$ at $(\eta+i|\eta|\mu)$ are given by,
\begin{align}
\label{eqn1}
|\eta|h_{00}-i\eta_1h_{10}-i\eta_2h_{20}-\frac{1}{2}|\eta|(h_{00}+h_{11}+h_{22})=0\\
\label{eqn2}
|\eta|h_{01}-i\eta_1h_{11}-i\eta_2h_{21}+\frac{1}{2}i\eta_1(h_{00}+h_{11}+h_{22})=0\\
\label{eqn3}
|\eta|h_{02}-i\eta_1h_{12}-i\eta_2h_{22}+\frac{1}{2}i\eta_2(h_{00}+h_{11}+h_{22})=0\\
\label{eqn4}
-2v+h_{11}=0\\
\label{eqn5}
h_{12}=0\\
\label{eqn6}
-2v+h_{22}=0\\
\label{eqn7}
-\frac{1}{2}|\eta|(h_{11}+h_{22})-i\eta_1h_{10}-i\eta_2h_{20}+2|\eta|v=0\\
\label{eqn8}
-|\eta|\sigma=0
\end{align}
To prove $B_1$ is non-degenerate, we show that the solution $(h_{\alpha\beta},v,\sigma)$ to the system above must be zero. According to equations $(3.5)$,$(3.6)$ and $(3.7)$, we can replace $h_{11}$ and $h_{22}$ by $2v$ and $h_{21}$ by $0$. Then equation $(3.8)$ gives 
$$2|\eta|v-(i\eta_1h_{10}+i\eta_2h_{20})-2|\eta|v=0,$$
i.e.$$(i\eta_1h_{10}+i\eta_2h_{20})=0.$$
Equation $(3.2)$ gives:
$$\frac{1}{2}|\eta|h_{00}-(i\eta_1h_{10}+i\eta_2h_{20})-2|\eta|v=0.$$
It follows that $$h_{00}=4v.$$
Multiplying $(3.3)$ by $(i\eta_1)$ and $(3.4)$ by $(i\eta_2)$, and then summing them, we obtain,
$$2|\eta|^2v+4|\eta|^2v=0.$$
Thus $v=0$ and consequently $h_{ij}=0$ for all $0\leq i,j\leq 2$. 
\\Finally, it's obvious from equation (3.9) that $\sigma=0$. This completes the proof.

\end{proof}
$Remark.$ One can see from the proof above that principal symbols of the operator $P_1$ are simple so that ellipticity follows from a direct verification of the conditions $(A)$ and $(B)$. However, in the divergence-gauge case, principal symbols of the operator $P_2$ are too complicated for us to carry out the same computation as above. In the following, we will use an intermediate operator which has Bianchi gauge term $G_1$ in the interior operator and divergence gauge $G_2$ in the boundary operator, to approach the ellipticity of the operator $P_2$.
\subsection{Ellipticity with divergence gauge}$~~$

Define a new operator $\hat P=(L_1,B_2)$, i.e. $\hat P$ consists of the interior operator $L$ as in (2.10) with $Z(h)=0$ and the boundary operator $B$ as in (2.11) with $G'_h=\delta_{g}h$.
Then we have the following lemma.
\begin{lemma} 
The differential operator $\hat P$ is elliptic.
\end{lemma}
\begin{proof}
This can be proved by a slight modification of the previous proof. Notice $\hat P$ has the same interior operator as $P_1$, and its boundary operator differs from that of $P_1$ only by the gauge term. So changing $\beta_gh$ to $\delta_gh$ (they only differ by a trace term) in the boundary operator, the new principal symbols of the  boundary operator at $\xi=(i|\eta|,\eta_1,\eta_2)$ become
\begin{align}
\label{eqn8}
|\eta|h_{00}-i\eta_1h_{10}-i\eta_2h_{20}=0\\
\label{eqn8}
|\eta|h_{01}-i\eta_1h_{11}-i\eta_2h_{21}=0\\
\label{eqn8}
|\eta|h_{02}-i\eta_1h_{12}-i\eta_2h_{22}=0\\
\label{eqn8}
-2v+h_{11}=0\\
\label{eqn8}
h_{12}=0\\
\label{eqn8}
-2v+h_{22}=0\\
\label{eqn8}
-\frac{1}{2}|\eta|(h_{11}+h_{22})-i\eta_1h_{10}-i\eta_2h_{20}+2|\eta|v=0\\
\label{eqn8}
-|\eta|\sigma=0
\end{align}
By equations $(3.13)$,$(3.14)$ and $(3.15)$, we can replace $h_{11}$ and $h_{22}$ by $2v$ and $h_{21}$ by $0$. Then $(3.16)$ gives 
$$2|\eta|v-(i\eta_1h_{10}+i\eta_2h_{20})-2|\eta|v=0,$$
i.e.$$(i\eta_1h_{10}+i\eta_2h_{20})=0.$$
Consequently, equation $(3.10)$ yields $h_{00}=0$.
\\Multiplying $(3.11)$ by $(i\eta_1)$ and $(3.12)$ by $(i\eta_2)$, and then summing, we obtain
$$2|\eta|^2v=0.$$
Thus $v=0$ and consequently $h_{ij}=0$ for all $0\leq i,j\leq 2$. 

\end{proof}
Next, we use the method exploit in [AK] to prove ellipticity for the operator $P_2$.
 \begin{theorem}
 The operator $P_2$ is elliptic.
 \end{theorem}
\begin{proof}
The ellipticity of a general operator $P=(L,B)$ is equivalent to the existence of a uniform estimate:
\begin{equation}
||(h,v,\sigma)||_{C^{m,\alpha}}\leq C(||L(h,v,\sigma)||_{C^{m-k,\alpha}} +||B(h,v,\sigma)||_{C^{m-j,\alpha}}+||(h,v,\sigma)||_{C^0}),
\end{equation}
together with such an estimate for the adjoint operator. In the expression above, $k$ and $j$ denote the order of derivative in the principal parts of the operators $L$ and $B$. 

The operator $P_2$ is then elliptic as a consequence of the following two facts, which are proved in Lemma 3.4 and Proposition 3.5 below.
\\(1) The inequality $(3.18)$ holds for $P_2$;
\\(2) The operator $P_2$ is formally self-adjoint.

\end{proof}
\begin{lemma}
Inequality $(3.18)$ holds for $P_2$, i.e.
\begin{equation}
||(h,v,\sigma)||_{C^{m,\alpha}}\leq C(||L_2(h,v,\sigma)||_{C^{m-2,\alpha}} +||B_2(h,v,\sigma)||_{C^{m-j,\alpha}}+||(h,v,\sigma)||_{C^0}).
\end{equation}
\end{lemma}
\begin{proof}
By Lemma 3.2, the inequality $(3.19)$ must hold if $L_2$ is replaced by $L_1$, i.e.
$$||(h,v,\sigma)||_{C^{m,\alpha}}\leq C(||L_1(h,v,\sigma)||_{C^{m-2,\alpha}} +||B_2(h,v,\sigma)||_{C^{m-j,\alpha}}+||(h,v,\sigma)||_{C^0}).$$
Observe that $L_1(h,v,\sigma)=L_2(h,v,\sigma)+\big(D^2(trh)+\Delta_g (trh)g+(\delta\delta h)g , 0,0\big)$. Thus the inequality (3.19) will be true if one can control the terms $\delta\delta h$ and $D^2(trh)$. So it suffices to prove
\begin{equation}
||\delta h||_{C^{m-1,\alpha}}\leq C(||L_2(h,v,\sigma)||_{C^{m-2,\alpha}} +||B_2(h,v,\sigma)||_{C^{m-j,\alpha}}+||(h,v,\sigma)||_{C^0},
\end{equation}
and
\begin{equation}
||D^2trh||_{C^{m-2,\alpha}}\leq C(||L_2(h,v,\sigma)||_{C^{m-2,\alpha}} +||B_2(h,v,\sigma)||_{C^{m-j,\alpha}}+||(h,v,\sigma)||_{C^0}).
\end{equation}
First notice that, $L_2(h)$ is the $2^{nd}$order part of the linearization of the map:
$$\Phi(g)=Ric_g+\delta^{\ast}G_2+T_2-2du\otimes du-2e^{-4u}d\phi\otimes d\phi.$$
From the proof of Lemma 2.4, one sees that
$$\delta\Phi(g)+2(\Delta_g u-2e^{-4u}|d\phi|^2)du+2e^{-4u}(\Delta_g\phi+4\langle du,d\phi\rangle)d\phi=\delta\delta^{\ast}\delta_{\tilde{g}}g.$$
Assume $g$ is a zero of $\Phi$. Linearizing the above equation at $\tilde{g}=g$ with respect to $h$ gives
$$\delta D\Phi(h)+O_0=\delta\delta^{\ast}(\delta h),$$
where $O_0$ denotes terms of $0$-derivative order with respect to $h$. In the equation above, it is of derivative order 3 on the right hand side, so the left hand side must be also of order 3, hence we obtain,
$$\delta\delta^{\ast}(\delta h)=\delta L_2(h).$$
The operator $\delta\delta^{\ast}$ is elliptic with respect to Dirichlet boundary conditions, and $\delta h$ is included in the boundary operator $B_2$. Thus inequality $(3.20)$ holds.

To prove inequality $(3.21)$, we use the Gauss equation at $\partial S$ :
$$|A_{g}|^2-H_g+s_{g^T}=s_g-2Ric_g(\mathbf n,\mathbf n),$$
where $A_g$ is the second fundamental form of $\partial S\subset (S,g)$ and $s_{g^T}$ is the scalar curvature of the metric $g^T$ on $\partial S$. The right side of the equation above can be converted as 
\begin{equation*}
\begin{split}
&s_g-2Ric_g(\mathbf n, \mathbf n)=-2[Ric_g(\mathbf n,\mathbf n)-\frac{1}{2}s_g\cdot g(\mathbf n,\mathbf n)]\\
&=-2[Ric_g-\frac{1}{2}s_g\cdot g+\delta^*\delta_{\tilde g}g](\mathbf n,\mathbf n)+2\delta^*\delta_{\tilde g}g(\mathbf n,\mathbf n).
\end{split}
\end{equation*}
Notice that linearization of the term $2[Ric_g-\frac{1}{2}s_g\cdot g+\delta^*\delta_{\tilde g}g]$ is $L_2$. So we take linearization of the Gauss equation and obtain,
$$(|A_g|^2-H_g+s_{g^T})'_{h}=-L_2(\mathbf n,\mathbf n)+2\delta^{\ast}\delta h(\mathbf n,\mathbf n)+O_1,$$
where $O_1$ denotes terms of derivative order no higher than $1$ with respect to $h$. Observe that the right side of the equation above is well controlled because of (3.20). On the left side, $s_{g^T}'=\Delta_{g^T}(trh^T)+\delta\delta(h^T)+O_1$ and the terms $A'_h$, $H'_h$ only involve first order derivatives in $h$ so they can be ignored according to the interpolation inequality,
$||h||_{C^{m-1,\alpha}}\leq \epsilon ||h||_{C^{m,\alpha}}+\epsilon^{-1}||h||_{C^0}$. Thus we get control of the term $(\Delta_{g^T}(trh^T)+\delta\delta(h^T))$.
Writing $h^T=B_{2,0}+2vg^T$, where $B_{2,0}=h^T-2vg^T$ is one of the boundary conditions in $B_2$, it follows that 
$$s_{g^T}'=\Delta_{g^T}(trh^T-2v)+\delta\delta(B_{2,0})+O_1$$ 
Therefore, by the ellipticity of the Laplace operator on $\partial S$, we obtain the estimate for $(trh^T-2v)$ along $\partial S$:
\begin{equation}
||(trh^T-2v)|_{\partial S}||_{C^{m,\alpha}}\leq C(||L_2(h,v,\sigma)||_{C^{m-2,\alpha}} +||B_2(h,v,\sigma)||_{C^{m-j,\alpha}}+||h||_{C^0}).
\end{equation}

Since the term $(h^T-2vg^T)$ is included in the boundary operator, $tr(h^T-2vg^T)=(trh^T-4v)|_{\partial S}$ is also controlled. Comparing with $(3.22)$, we obtain the control for $v$ on $\partial S$,
$$||v|_{\partial S}||_{C^{m,\alpha}}\leq C(||L_2(h,v,\sigma)||_{C^{m-2,\alpha}} +||B_2(h,v,\sigma)||_{C^{m-j,\alpha}}+||h||_{C^0}).$$
In addition, $\Delta_g v$ is one of the components of the interior operator $L_2$, so from the ellipticity of Laplace operator with Dirichlet boundary condition, we obtain the uniform estimate for $v$ over $S$, 
$$||v||_{C^{m,\alpha}}\leq C(||L_2(h,v,\sigma)||_{C^{m-2,\alpha}} +||B_2(h,v,\sigma)||_{C^{m-j,\alpha}}+||(h,v,\sigma)||_{C^0}).$$

Furthermore, observe that $\Delta_g\sigma$ is the last component of the interior operator $L_2$ and $\mathbf n(\sigma)$ is one of the boundary terms in $B_2$. Thus, based on the ellipticity of Laplace operator with Neumann boundary condition, we also have the uniform estimate for $\sigma$ over $S$:
$$||\sigma||_{C^{m,\alpha}}\leq C(||L_2(h,v,\sigma)||_{C^{m-2,\alpha}} +||B_2(h,v,\sigma)||_{C^{m-j,\alpha}}+||(h,v,\sigma)||_{C^0}).$$

Now with $v,\sigma$ being well controlled, inequality $(3.21)$ is equivalent to the following one,
\begin{equation*}
\begin{split}
||D^2trh||_{C^{m-2,\alpha}}\leq C(&||L_2(h)||_{C^{m-2,\alpha}} +||\delta(h)|_{\partial S}||_{C^{m-1,\alpha}}+||h^T|_{\partial S}||_{C^{m,\alpha}}\\
&+||H'_h|_{\partial S}||_{C^{m-1,\alpha}}+||h||_{C^0}).
\end{split}
\end{equation*}
This estimate is proved in Lemma 3.2 of $[AK]$. So we complete the proof of the uniform estimate.

\end{proof}
\begin{proposition}
Let $\mathcal M_2$ be the space of data $(h,v,\sigma)$ which is in the kernel of the boundary operator $B_2$, i.e.
\begin{equation}
\begin{split}
\mathcal M_2=\{\quad(h,v,\sigma)\in&[(S_2)_{\delta}^{m,\alpha}\times C_{\delta}^{m,\alpha}\times C_{\delta}^{m,\alpha}](S):\\
&\begin{cases}
\delta_g(h)=0,\\
~h^T-2vg^T=0,\\
~H'_h-2\mathbf n(v)=0,\\
~\mathbf n(\sigma)=0,
\end{cases}\quad\text{on}~\partial S\quad\}.
\end{split}
\end{equation}
Then the operator $L_2: \mathcal M_2\rightarrow (S_2)_{\delta}^{m-2,\alpha}\times C_{\delta}^{m-2,\alpha}\times C_{\delta}^{m-2,\alpha}](S)$, given by
\begin{equation*}
\begin{split}
L_2(h,v,w)&=\{~D^{\ast}Dh-Z_2(h),~~8\Delta_g v,~~8e^{-4u}[\Delta_g \sigma+4\langle du,d\sigma\rangle]~\}\\
&=\{~D^{\ast}Dh-D^2(trh)-\Delta_g(trh)g-(\delta\delta h)g,~~8\Delta_g v,~~8e^{-4u}[\Delta_g \sigma+4\langle du,d\sigma\rangle]~\},
\end{split}
\end{equation*}
is formally self-adjoint.
\end{proposition}
\begin{proof}
We will prove this proposition by showing that $L_2$ arises as the $2^{\text{nd}}$ variation of a natural variational problem on the data $(g,u,\phi)$.

To begin, the Einstein equation $Ein_{g^{(4)}}=0$ is the functional derivative of the Einstein-Hilbert action 
$$I_{\text{EH}}=\int_{V^{(4)}}R_{g^{(4)}}dvol_{g^{(4)}}.$$ 
Reducing this action from 4-dimensional spacetime $V^{(4)}$ to the 3-dimensional quotient space $S$, one obtains the following functional on the data $(g,u,\phi)$, of which the Euler-Lagrange equations are exactly the field equations (II) in $\S2$,  
$$I_{\text{eff}}=\int_{S} s-2|du|^2-2e^{-4u}|d\phi|^2 dvol_g.$$ 
We refer to [H1][H2] for further discussion of the action $I_{\text{eff}}$.

Since the boundary $\partial S$ is nonempty, {as is well known} it is necessary to add boundary terms to the action {such as Gibbons-Hawking boundary terms, cf.[GH]}. The right action with boundary terms in our case is given by
$$I=\int_{S} s-2|du|^2-2e^{-4u}|d\phi|^2 dvol_g+2\int_{\partial S}H_gdvol_{g^T}+16\pi m_{\text{ADM}}(g).$$
Here the last term $m_{\text{ADM}}(g)$ is the ADM mass of the asymptotically flat manifold $(S,g)$.

Next, let $(\mathbf E,\mathbf F,\mathbf H)$ be the equations in the system (II), i.e. 
\begin{equation}
\begin{split}
&\mathbf E[(g,u,\phi)]=\frac{1}{2}(s-2|du|^2-2e^{-4u}|d\phi|^2)g-Ric_g+2(du)^2+2e^{-4u}(d\phi)^2,\\ 
&\mathbf F[(g,u,\phi)]=-4\Delta_g u+8e^{-4u}|d\phi|^2,\\
&\mathbf H[(g,u,\phi)]=-4e^{-4u}(\Delta_g\phi+4\langle du,d\phi\rangle).
\end{split}
\end{equation}
Then the variation of $I$ at $g$ with respect to the infinitesimal deformation $h$ is 
\begin{equation}
\begin{split}
I'_g(h)&=\int_{S}\langle \mathbf E,h\rangle dvol_g +\int_{\partial S} -\langle A,h\rangle+Htrh^T dvol_{g^T}\\
&\quad\quad\quad
-\int_{\partial S_{\infty}} \mathbf n(trh)+\delta h(\mathbf n)dvol_{\partial S_{\infty}}+16\pi(m_{ADM}(g))'_h.
\end{split}
\end{equation}
We refer to \S5.3 for the details of the computation. To abbreviate notation, we shall omit the volume form in the following.

Notice that the terms in the second line of the equation $(3.25)$ can be removed, because we have 
$$\int_{\partial S_{\infty}} \mathbf n(trh)+\delta h(\mathbf n)=16\pi(m_{ADM}(g))'_h,$$
based on the definition of ADM mass and its variation, cf.[RT],[B1].

Basic computation shows the variations of $I$ with respect to deformations $v,\sigma$ of $u$ and $\phi$ are given by,
\begin{equation}
I'_u(v)=\int_{S}\langle \mathbf F, v\rangle +\int_{\partial S} -4\mathbf n(u)v~+\int_{\partial S_{\infty}} -4\mathbf n(u)v,
\end{equation}
and
\begin{equation}
I'_{\phi}(\sigma)=\int_{S}\langle \mathbf H,\sigma\rangle +\int_{\partial S} -4e^{-4u}\mathbf n(\phi)\sigma~+\int_{\partial S_{\infty}} -4e^{-4u}\mathbf n(\phi)\sigma.
\end{equation}
By simply checking the decay rate, one sees easily that the boundary terms at infinity in the expressions above are both zero.

Now let $(g,u,\phi)$ be a triple such that $(\mathbf E,\mathbf F,\mathbf H)[(g,u,\phi)]=0$, and take a 2-parameter varation of data $(g_{st},u_{st},\phi_{st})=(g,u,\phi)+s(h,v,\sigma)+t(k,w,\zeta)$, with infinitesimal deformations $(h,v,\sigma),(k,w,\zeta)\in\mathcal M_2$.

Based on the boundary conditions in the expression $(3.23)$, we have $h^T=2vg^T$. The equation $(3.25)$ then becomes:
\begin{equation}
I'_g(h)=\int_{S}\langle \mathbf E,h\rangle +\int_{\partial S} 2vH.
\end{equation}
Take one more variation of the equation $(3.28)$ with respect to $k$, and we obtain
\begin{equation}
\begin{split}
I''_g(h,k)=\int_{S}\langle \mathbf E'_k,h\rangle +\int_{\partial S} 2vH'_k+4vw H.
\end{split}
\end{equation}
Similar operation of the equations $(3.26)$ and $(3.27)$ yields,
\begin{equation}
I''_u(v,w)=\int_{S}\langle \mathbf F'_{w}, v\rangle +\int_{\partial S} -4\mathbf n(w)v,
\end{equation}
and
\begin{equation}
I''_{\phi}(\sigma,\zeta)=\int_{S}\langle \mathbf H'_{\zeta},\sigma\rangle +\int_{\partial S} -4e^{-4u}\mathbf n(\zeta)\sigma.
\end{equation}

From the symmetry of second variation, we know that $I''(h,k)=I''(k,h)$, $I''(w,v)=I''(v,w)$ and $I''(\sigma,\zeta)=I''(\zeta,\sigma)$. The equations $(3.29-31)$ then imply that:
\begin{equation*}
\begin{split}
&\int_{S}[\langle \mathbf E'_k,h\rangle +\langle \mathbf F'_{w},v\rangle +\langle \mathbf H'_{\zeta},\sigma\rangle] +\int_{\partial S} [2vH'_k+4vw H-4\mathbf n(w)v-4e^{-4u}\mathbf n(\zeta)\sigma]\\
=&
\int_{S}[\langle \mathbf E'_h,k\rangle + \langle \mathbf F'_v,w\rangle +\langle \mathbf H'_{\sigma},\zeta\rangle] +\int_{\partial S} [2w H'_h+4vw H-4\mathbf n(v)w-4e^{-4u}\mathbf n(\sigma)\zeta].
\end{split}
\end{equation*}
By the boundary condition in $(3.23)$, $H'_h-2\mathbf n(v)=0~\mathbf n(\sigma)=0$ on $\partial S$, and the same for $(k,w,\zeta)$. Thus we can remove the boundary terms above and obtain
\begin{equation}
\begin{split}
\int_{S}[\langle  \mathbf E'_k,h\rangle+\langle  \mathbf F'_{w},v\rangle+\langle  \mathbf H'_{\zeta},\sigma\rangle] =\int_{S}[\langle  \mathbf E'_h,k\rangle+\langle  \mathbf F'_v,w\rangle+\langle  \mathbf H'_{\sigma},\zeta\rangle] .
\end{split}
\end{equation}

On the other hand, from the boundary condition $\delta h=\delta k=0$ on $\partial S$, it follows that,
\begin{equation}
\int_S\langle \delta^{\ast}\delta k,h\rangle=\int_S\langle \delta k, \delta h\rangle=\int_S\langle \delta^{\ast}\delta h,k\rangle.
\end{equation}

Combining equations $(3.32)$ and $(3.33)$, we obtain,
\begin{equation}
\int_{S}\langle ( \mathbf E'_k-\delta^{\ast}\delta k,  \mathbf F'_{w}, \mathbf H'_{\zeta}),(h,v,\sigma)\rangle =\int_{S}\langle  \mathbf (\mathbf E'_h-\delta^{\ast}\delta h,  \mathbf F'_v,  \mathbf H'_{\sigma}),(k,w,\zeta)\rangle .
\end{equation}
Notice that the terms of second order and first order derivative in $ (\mathbf E'_h-\delta^*\delta h,  \mathbf F'_v,  \mathbf H'_{\alpha})$ are the same as in the operator $-\frac{1}{2}L_2$; and the zero order terms in the equation $(3.34)$ can be removed because of symmetry. Therefore it follows that
$$\int_S\langle L_2(k,w,\zeta),(h,v,\sigma)\rangle=\int_S\langle L_2(h,v,\sigma),(k,w,\zeta)\rangle,$$
which proves the formal self-adjointness of the operator $P_2$.

\end{proof}
Ellipticity of the operator $P_2$ implies that the boundary value problem $(2.7)$ with the divergence gauge is elliptic. Together with the {local} equivalence between the {sets} $\mathbf{Z}_C$ and $\mathcal{E}_C$ in \S 2.3, we conclude that 
the collection of boundary conditions $(2.6)$ is elliptic for the stationary vacuum field equations (II) modulo diffeomorphisms in $\mathcal D^{m+1,\alpha}_0(S)$.
\subsection{Back to $g_S$}
It is now basically trivial to prove the ellipticity of the system $(1.6)$ equipped with boundary conditions $(1.7)$, using the result we have obtained.

First observe that, by combining the first and second equations in $(1.6)$, the system is equivalent to the following one, 
\begin{equation}
\begin{cases}
Ric_{g_S}-D^2u-(du)^2-2e^{-4u}(d\phi\otimes d\phi-|d\phi|^2g_S)\\
\quad\quad\quad\quad\quad\quad\quad\quad\quad+(\Delta_{g_S} u-|du|^2-2e^{-4u}|d\phi|^2)g_S=0,\\
\Delta_{g_S} u-|du|^2-2e^{-4u}|d\phi|^2=0,\\
\Delta_{g_S}\phi+3 \langle du,d\phi\rangle=0.
\end{cases}
\end{equation}
The trace of the first equation above is given by
$$s_{g_S}+4\Delta_{g_S} u-4|du|^2-2e^{-4u}|d\phi|^2.$$
Let $T_S$ be as, $T_S=-\frac{1}{2}(s_{g_S}+4\Delta_{g_S} u-4|du|^2-2e^{-4u}|d\phi|^2)g_S$. Let $G_S$ be the pull back by conformal transformation of the divergence gauge term $\delta_{\tilde g}g$, i.e. $G_S=\delta_{e^{2u}\tilde g_S}(e^{2u}g_S),$
where $\tilde g_S$ is a reference metric near $g_S$.

Inserting $(T_S+\delta^*_{e^{2u}g_S}G_S)$ to the first equation in system $(3.35)$, we obtain
\begin{equation}
\begin{cases}
Ric_{g_S}-D^2u-(du)^2-2e^{-4u}(d\phi\otimes d\phi-|d\phi|^2g_S)\\
\quad\quad\quad+(\Delta_{g_S} u-|du|^2-2e^{-4u}|d\phi|^2)g_S+T_S+\delta^*_{e^{2u}g_S}G_S=0,\\
\Delta_{g_S} u-|du|^2-2e^{-4u}|d\phi|^2=0,\\
\Delta_{g_S} \phi+3 \langle du,d\phi \rangle=0.
\end{cases}
\end{equation}
According to the system above, we define a differential operator $\mathcal{P}_S=(\mathcal{L}_S,\mathcal{B}_S)$, which consists of the interior operator $\mathcal{L}_S$, mapping the data $(g_S,u,\phi)$ to the interior expressions in $(3.36)$, given by
\begin{equation*}
\begin{split}
\mathcal{L}_S:[Met_{\delta}^{m,\alpha}\times C^{m,\alpha}_{\delta}\times &C^{m,\alpha}_{\delta} ](S)
\rightarrow [(S_2)^{m-2,\alpha}_{\delta+2}\times C^{m-2,\alpha}_{\delta+2}\times C^{m-2,\alpha}_{\delta+2}](S)\\
\mathcal{L}_S(g_S,u,\phi)=\{~&2[Ric_{g_S}-D^2u-(du)^2-2e^{-4u}(d\phi\otimes d\phi-|d\phi|^2g_S)\\
&\quad\quad\quad\quad+(\Delta_{g_S} u-|du|^2-2e^{-4u}|d\phi|^2)g_S+T_S+\delta^*_{e^{2u}g_S}G_S],\\
&8e^{-2u}(\Delta_{g_S} u-|du|^2-2e^{-4u}|d\phi|^2),\\
& 8e^{-6u}(\Delta_{g_S} \phi+3 \langle du,d\phi \rangle)~\};\\
\end{split}
\end{equation*}
and the boundary operator $\mathcal B_S$, mapping the data $(g_S,u,\phi)$ to boundary data including the gauge term $G_S$ and the terms in $(1.7)$, given by
\begin{equation*}
\begin{split}
&\mathcal{B}_S:~[Met_{\delta}^{m,\alpha}\times C^{m,\alpha}_{\delta}\times C^{m,\alpha}_{\delta}](S)\to \mathbf B^{m,\alpha}(S),\\
&\quad\quad\mathcal{B}(g,u,\omega)=
\{~G_S,\quad g_S^T|_{\partial S}-\gamma,\quad H_{g_S}-\lambda,\quad \mathbf n_{g_S}(\phi)-f~\}.
\end{split}
\end{equation*}
In addition, define an operator $\mathcal{Q}$ as the conformal transformation in $(2.2)$,
\begin{equation*}
\begin{split}
\mathcal Q:[Met_{\delta}^{m,\alpha}\times C^{m,\alpha}_{\delta}\times& C^{m,\alpha}_{\delta} ](S)\rightarrow[Met_{\delta}^{m,\alpha}\times C^{m,\alpha}_{\delta}\times C^{m,\alpha}_{\delta} ](S)\\
&\mathcal Q(g_S,u,\phi)=(e^{2u}g_S,u,\phi),
\end{split}
\end{equation*}
It is easy to see, by elementary computation, that the operator $\mathcal P_S$ is exactly the composition of $\mathcal Q$ and $\mathcal P_2$ defined in $\S2$, i.e.
\begin{equation*}
\mathcal P_S=\mathcal P_2\circ \mathcal Q.
\end{equation*}
The operator $\mathcal P_2$ has already been proved to be elliptic and $\mathcal Q$ is obviously an isomorphism. As a consequence, the operator $\mathcal P_S$ is also elliptic. This gives the proof of Theorem 1.1.
\begin{flushright}
$\square$
\end{flushright}

{In the following section, we will apply the ellipticity of $\mathcal P_2$ to prove the manifold theorem for the moduli space $\mathcal E_C$ of stationary vacuum spacetimes.}
\section{Manifold Theorem}
Fix a triple $(\tilde g,\tilde u,\tilde\phi)$ which solves (II) on $S$ and satisfies the following conditions on the boundary $\partial S$
$$\begin{cases}
e^{-2u}g^T|_{\partial S}=\gamma\\
H-2\mathbf n(u)=e^{-u}\lambda\\
\mathbf n(\phi)=e^{2u}f.
\end{cases}$$
Here $\gamma,\lambda$ and $f$ are prescribed fields on $\partial S$. Given these information, we define the following Banach spaces.
\begin{definition}
\begin{equation*}
\begin{split}
\mathcal M_S&=\{~(g,u,\phi)\in [Met_{\delta}^{m,\alpha}\times C^{m,\alpha}_{\delta}\times C^{m,\alpha}_{\delta}](S):
\begin{cases}
\delta_{\tilde g}g=0,\\
e^{-2u}g^T|_{\partial S}=\gamma\\
H-2\mathbf n(u)=e^{-u}\lambda\\
\mathbf n(\phi)=e^{2u}f
\end{cases}\text{on }\partial S\};\\
\mathcal M_C&=\{~(g,u,\phi)\in [Met_{\delta}^{m,\alpha}\times C^{m,\alpha}_{\delta}\times C^{m,\alpha}_{\delta}](S):\quad\delta_{\tilde g}g=0 ~on~ \partial S~\}
 \end{split}
 \end{equation*}
\end{definition}
In the definition of $\mathcal M_S$, we replace the previous boundary condition $\mathbf n(\phi)=e^{-u}f$ in (2.6) by $\mathbf n(\phi)=e^{2u}f$, to ensure that the operator $D\hat\Phi$ below is formally self-adjoint on the tangent space $T\mathcal M_S$. This does not affect the elliptic property of the operator.

Define a map:
$$\Phi:~\mathcal M_C\rightarrow [(S_2)^{m-2,\alpha}_{\delta+2}\times C^{m-2,\alpha}_{\delta+2}\times C^{m-2,\alpha}_{\delta+2}](S)$$
$$\Phi(g,u,\phi)=(~\mathbf E-\delta^*\delta_{\tilde g}g,\quad\mathbf F,\quad\mathbf H~)$$
where the terms $\mathbf E,\mathbf F,\mathbf H$ are defined as in $(3.24)$.
Thus, the zero set $\Phi^{-1}(0)$ consists of stationary vacuum data $(g,u,\phi)$ satisfying $\delta_{\tilde g}g=0$ on $S$, i.e.
$\Phi^{-1}(0)=\mathbf{Z}_C,$ where $\mathbf Z_C$ is as in Definition 2.3. Obviously, $(\tilde g,\tilde u,\tilde\phi)\in\mathbf{Z}_C$. Based on the analysis in \S 2.3, to prove the moduli space $\mathcal E_C$ admits Banach manifold structure, it suffices to prove the zero set $\Phi^{-1}(0)$ is a smooth Banach manifold near $(\tilde g,\tilde u,\tilde \phi)$. The main step of this is the following theorem.
\begin{theorem}
At the point $(\tilde g,\tilde u,\tilde \phi)\in \Phi^{-1}(0)$, the linearization $D\Phi$ is surjective and its kernel splits in $T\mathcal M_C$.
\end{theorem}

\subsection{Proof of Theorem 4.2}$~~$

Surjectivity can be proved in a similar way as in $[A1]$ and $[AK]$. 
Let $D\hat \Phi$ be the restriction of $D\Phi$ to the subspace $T\mathcal M_S\subset T\mathcal M_C$, i.e. 
$$D\hat \Phi=D\Phi|_{T\mathcal M_S}.$$
Then the operator $D\hat \Phi$ is Fredholm by Theorem 3.3 and of index 0 because it is formally self-adjoint (cf.\S5.4). So $\text{Im}(D\hat\Phi)$ is closed in $ [(S_2)^{m-2,\alpha}_{\delta+2}\times C^{m-2,\alpha}_{\delta+2}\times C^{m-2,\alpha}_{\delta+2}](S)$ and its cokernel  is of the same dimension as the kernel $K=\text{Ker}D\hat{\Phi}$. If $K$ is trivial, then $D\hat\Phi$ is surjective, {and hence} so is $D\Phi$. 

Suppose $K$ is nontrivial. From the self-adjointness, it follows that for any element $(k,w,\zeta)\in K$ and $(h,u,\sigma)\in T\mathcal M_{S}$,
$$\int_S\langle D\hat\Phi(h,v,\sigma),(k,w,\zeta)\rangle=\int_S\langle (h,v,\sigma),D\hat\Phi(k,w,\zeta)\rangle=0.$$ 
Thus $\text{Im}(D\hat\Phi)=K^{\perp}$ with respect to the $L^2$ inner product. To prove surjectivity of $D\Phi$, it suffices to prove that for any triple $(k,w,\zeta)\in K$, there exists an element  $(h,v,\sigma)\in T\mathcal M_{C}$, such that $\int_{S}\langle D\Phi (h,v,\sigma),(k,w,\zeta)\rangle\neq 0$. 

Assume this is not true, i.e. there exists a nonzero element $(k,w,\zeta)\in K$ such that,
\begin{equation}
\int_{S}\langle D\Phi (h,v,\sigma),(k,w,\zeta)\rangle=0,\quad\forall (h,v,\sigma)\in T\mathcal M_{C}.
\end{equation} 
We claim that the triple $(k,w,\zeta)$ must satisfy the following system,
\begin{equation}
\begin{split}
&\begin{cases}
D\Phi(k,w,\zeta)=0,\\
\delta k=0,
\end{cases}~\text{on $S$}\\
&\begin{cases}
k^T=0,\\
(A'_k)^T=0,\\
w=\zeta=0,\\
\mathbf n(w)+\mathbf n'_k(\tilde u)=0,\\
\mathbf n(\zeta)+\mathbf n'_k(\tilde\phi)=0.
\end{cases} \text{on $\partial S$.}
\end{split}
\end{equation}
Proof of the claim:
First choose $(h,v,\sigma)=(\delta^*X,L_Xu,L_X\phi)$ in (4.1), for some vector field $X$ which vanishes on $\partial S$. Such a deformation corresponds to varying the data $(\tilde g,\tilde u,\tilde\phi)$ in the tangential direction of the action of $\mathcal D_0^{m+1,\alpha}(S)$. In this case, since the stationary vacuum field equations (II) are invariant under diffeomorphisms, it follows that
$$D\Phi(\delta^*X,L_Xu,L_X\phi)=(\delta^*Y,0,0)\quad\text{at}~(\tilde g,\tilde u,\tilde \phi),$$ 
where $Y=\delta\delta^*X$.
Note that Lemma 2.5 shows the operator $\delta\delta^*$ is surjective, so $Y\in T^{m-2,\alpha}_{\delta+2}(S)$ can be arbitrarily prescribed. Moreover, the fact $(h,v,\sigma)\in T\mathcal M_C$ implies that  $\delta h=0$ on $\partial S$, so that $Y=0$ on $\partial S$.  It follows from the equation $(4.1)$ that,
$$0=\int_{S}\langle\delta^* Y,k\rangle=\int_{S}\langle Y,\delta k\rangle+(\int_{\partial S}+\int_{\partial S_{\infty}})k(Y,\mathbf n)=\int_{S}\langle Y,\delta k\rangle.$$
Here the boundary integral over $\partial S$ vanish because $Y=0$ on $\partial S$ and the boundary integral at infinity $\partial S_{\infty}$ also vanish because of the decay behavior of $k$ and $Y$. On the right side of the equation above, $Y$ can be arbitrarily prescribed in the bulk, so we must have $\delta k=0$ on $S$, which yields the second equation in (4.2).

Next applying integration by parts to (4.1), we obtain
\begin{equation}
\int_S \langle D\Phi(k,w,\zeta), (h,v,\sigma)\rangle+\int_{\partial S}\tilde B[(h,v,\sigma),(k,w,\zeta)]=0.
\end{equation}
Since this holds for any $(h,v,\sigma)\in T\mathcal M_C$, it implies that
\begin{equation*}
D\Phi(k,w,\zeta)=0~\text{on}~S,
\end{equation*}
and 
\begin{equation}
\int_{\partial S}\tilde B[(h,v,\sigma),(k,w,\zeta)]=0,\quad\forall (h,v,\sigma)\in T\mathcal M_C.
\end{equation}
Here $\tilde B$ is an anti-symmetric bilinear form given by
\begin{equation*}
\tilde B [(h,v,\sigma),(k,w,\zeta)]=B[(h,v,\sigma),(k,w,\zeta)]-B[(k,w,\zeta),(h,v,\sigma)],
\end{equation*}
with
\begin{equation*}
\begin{split}
B[(h,v,\sigma),(k,w,\zeta)]&=-k(\delta h,\mathbf n)+\frac{1}{2}\{-\langle\nabla_{\mathbf n}h,k\rangle-k(\mathbf n,dtrh)+trk[\mathbf n(trh)+\delta h(\mathbf n)]\}\\
&\quad+[4\mathbf n(w)-4k(\mathbf n,d\tilde u)+2trk\mathbf n(\tilde u)]v\\
&\quad +4e^{-4\tilde u}\sigma [\mathbf n(\zeta)-k(\mathbf n,d\tilde \phi)+\frac{1}{2}trk\mathbf n(\tilde \phi)-4w \mathbf n(\tilde \phi)].
\end{split}
\end{equation*}
Again, in equation (4.3) the boundary integral at infinity $\int_{S_{\infty}}$, that obtained from the integration by parts, is zero because of the decay behavior of the boundary term $\tilde B$.

On the other hand, since $(k,w,\zeta)$ is in the kernel of $D\hat\Phi$, it must satisfy the following boundary conditions (obtained by linearizing the boundary conditions in $\mathcal M_S$),
\begin{equation}
\begin{cases}
\delta k=0\\
k^T-2w \tilde g^T=0\\
H'_k-2\mathbf n(w)-2\mathbf n'_k(\tilde u)+w(H-2\mathbf n(\tilde u))=0\\
\mathbf n(\zeta)+\mathbf n'_k(\tilde \phi)-2w \mathbf n(\tilde\phi)=0
\end{cases}
\quad\text{on}~\partial S.
\end{equation}
Since $h\in T\mathcal{M}_C$, we have $\delta h=0$ on $\partial S$. The same holds for $k$. So we can simplify the bilinear form $B$ by removing the divergence terms and obtain,
\begin{equation}
\begin{split}
B[(h,v,\sigma),(k,w,\zeta)]&=\frac{1}{2}\{-\langle\nabla_{\mathbf n}h,k\rangle-k(\mathbf n,dtrh)+trk\mathbf n(trh)\}\\
&\quad+[4\mathbf n(w)-4k(\mathbf n,d\tilde u)+2trk\mathbf n(\tilde u)]v\\
&\quad +4e^{-4\tilde u}\sigma [\mathbf n(\zeta)-k(\mathbf n,d\tilde \phi)+\frac{1}{2}trk\mathbf n(\tilde \phi)-4w \mathbf n(\tilde \phi)].
\end{split}
\end{equation}
Take a triple $(h,v,\sigma)$ such that $h=0,~\nabla_{\mathbf n} h=0$ and $\sigma=v=0$ on $\partial S$. Inserting it into equation $(4.4)$ yields
\begin{equation*}
\int_{\partial S} 4\mathbf n(v)w+4e^{-4u}\zeta \mathbf n(\sigma)=0.
\end{equation*}
The scalar fields $\mathbf n(v)$ and $\mathbf n(\sigma)$ can be arbitrary on $\partial S$. So this implies that
$ w=\zeta=0$ on $\partial S$, which is the third boundary equation in (4.2).

Consequently, based on the second equation in $(4.5)$, we also obtain the first boundary equation in (4.2), since
$k^T=w g^T=0$ on $\partial S$. Moreover, according to the last equation in $(4.5)$,
\begin{equation*}
\mathbf n(\zeta)+\mathbf n'_k(\tilde\phi)= 4w\mathbf n(\tilde\phi)=0\quad\text{on}~\partial S.
\end{equation*}
This is the last equation in (4.2).

Since $k^T=0$ on $\partial S$, the trace of $k$ on the boundary is $trk=k(\mathbf n,\mathbf n)$. Thus in the last line of equation $(4.6)$, the term $[\mathbf n(\zeta)-k(\mathbf n,d\tilde \phi)+\frac{1}{2}trk \mathbf n(\tilde \phi)-4w \mathbf n(\tilde\phi)]$ vanishes on $\partial S$ because of the following computation,
\begin{equation*}
\begin{split}
&\mathbf n(\zeta)-k(\mathbf n,d\tilde \phi)+\frac{1}{2}trk \mathbf n(\tilde \phi)-4w \mathbf n(\tilde\phi)\\
&=\mathbf n(\zeta)-k(\mathbf n,d\tilde \phi)+\frac{1}{2}k(\mathbf n,\mathbf n) \mathbf n(\tilde \phi)-4w \mathbf n(\tilde\phi)\\
&=\mathbf n(\zeta)+\mathbf n'_k(\tilde\phi)-4w \mathbf n(\tilde\phi)\\
&=0.
\end{split}
\end{equation*} 
Here the second equality is based on the variation formula of the unit normal vector $\mathbf n$, cf.\S5: $\mathbf n'_k=-k(\mathbf n)+\frac{1}{2}k(\mathbf n,\mathbf n)\mathbf n.$ 
In addition, we already prove $\zeta=0$ on the boundary. Therefore, the form $B$ can be simplified further by removing the last line in $(4.6)$ and becomes,
\begin{equation}
\begin{split}
B[(h,v,\sigma),(k,w,\zeta)]&=\frac{1}{2}\{-\langle\nabla_{\mathbf n}h,k\rangle-k(\mathbf n,dtrh)+trk\mathbf n(trh)\}\\
&\quad\quad+[4\mathbf n(w)-4k(\mathbf n,d\tilde u)+2trk\mathbf n(\tilde u)]v
\end{split}
\end{equation}

Choose a triple $(h,v,\sigma)$ so that  $h=0$ and $\nabla_{\mathbf n} h=0$ on $\partial S$ and plug it into equation $(4.4)$. It follows that,
\begin{equation*}
\int_{\partial S}[4\mathbf n(w)-4k(\mathbf n,d\tilde u)+2trk\mathbf n(\tilde u)]v=0,
\end{equation*}
Since the term $v$ can be arbitrarily prescribed on $\partial S$, one obtains 
\begin{equation}
4\mathbf n(w)-4k(\mathbf n,d\tilde u)+2trk\mathbf n(\tilde u)=0\quad\text{on}~\partial S.
\end{equation}
Since $-4k(\mathbf n,d\tilde u)+2trk\mathbf n(\tilde u)=4\langle-k(\mathbf n)+\frac{1}{2}k(\mathbf n,\mathbf n)\mathbf n, d\tilde u\rangle=4\mathbf n'_k(\tilde u)$, the equation above yields $\mathbf n(w)+\mathbf n'_k(\tilde u)=0$ on $\partial S$, which is fourth boundary equation in (4.2).
Combining this with the third equation in $(4.5)$, one obtains 
\begin{equation}
H'_k=0\quad\text{on}~\partial S.
\end{equation}
Based on equation $(4.8)$, we can simplify the form $B$ further into the following expression, 
\begin{equation}
\begin{split}
B[(h,v,\sigma),(k,w,\zeta)]&=\frac{1}{2}\{-\langle\nabla_{\mathbf n}h,k\rangle-k(\mathbf n,dtrh)+trk\mathbf n(trh)\}
\end{split}
\end{equation}
Consequently $(4.4)$ implies that the following equation holds for any $h\in T\mathcal M_C$,  
\begin{equation}
\begin{split}
\int_{\partial S}&\{-\langle\nabla_{\mathbf n}h,k\rangle-k(\mathbf n,dtrh)+trk\mathbf n(trh)\}\\
&-\{-\langle\nabla_{\mathbf n}k,h\rangle-h(\mathbf n,dtrk)+trh\mathbf n(trk)\}~=0.
\end{split}
\end{equation}
Equation (4.11) together with the boundary condition (4.9) and $k^T=0$ on $\partial S$ implies that, c.f.[AK], $(A'_k)^T=0$ on $S$. This completes the proof of our claim.
\begin{flushright}
$\Box$
\end{flushright}

The first equation in $(4.2)$ implies that the variation of $(\mathbf E-\delta^*\delta_{\tilde g}g, \mathbf F, \mathbf H)$ with respect to the deformation $(k,w,\zeta)$ vanishes, i.e.
 $$D(\mathbf E-\delta^*\delta_{\tilde g}g, \mathbf F, \mathbf H)_{(\tilde g,\tilde u,\tilde\phi)}(k,w,\zeta)=0,$$
 Together with the second equation in $(4.2)$, we observe that $(k,w,\zeta)$ is in fact a vacuum deformation, i.e. it makes the linearization of $(\mathbf E, \mathbf F, \mathbf H)$ at $(\tilde g,\tilde u,\tilde\phi)$ vanish:
\begin{equation}
D(\mathbf E, \mathbf F, \mathbf H)_{(\tilde g,\tilde u,\tilde\phi)}(k,w,\zeta)=0.
\end{equation}
Translate to the normal geodesic gauge
$$k\rightarrow \tilde k=k+\delta^*V,$$
where $V$ is a vector field which vanishes on the boundary and makes $\tilde k_{0i}=0$ on $\partial S$.
Then the boundary conditions in $(4.2)$ imply that the Cauchy data for $(k,w,\zeta)$ vanishes on $\partial S$. Implementing this idea, we can obtain the following unique continuation result. 
\begin{proposition} 
The trivial data $(k,w,\zeta)=0$ is the only solution to system $(4.2)$. 
\end{proposition}
We refer to $\S4.2$ for the proof. This is a contradiction with our assumption that $(k,w,\zeta)$ is nonzero. As a consequence, $D\Phi$ must be surjective. Then it is a standard fact that the kernel of $D\Phi$ splits, cf.$[A1]$. This completes the proof of Theorem 4.2.

\subsection{Proof of Proposition 4.3}$~~$

The proof below is a generalization of the unique continuation result of [AH] from Riemannian Einstein metrics to sationary Lorentzian Einstein metrics. We first recall some basic facts about H-harmonic coordinates. We refer to [AM],[AH] for more details.

Let $C=I\times B^2$, where $I=[0,1]$ and $B^2$ is the unit disk in $\mathbb R^2$. Given a general metric $g$ on $C$, one can always choose the H-harmonic coordinates $\{\tau,x^i\} (\tau\geq 0, i=1,2)$ for $(C,g)$ such that level set $\{\tau=0\}$ coincides with the horizontal boundary $\partial_0 C=\{0\}\times B^2$. Using such coordinates, we can write the metric as,
$$g=z d\tau^2+\gamma_{ij}(\psi^id\tau+dx^i)(\psi^jd\tau+dx^j).$$
Here $\gamma$ is the induced metric on the level sets of $\tau$ function, $z$ is called the lapse function and $\psi$ is the shift vector. In addition, by expressing the Ricci tensor $Ric_g$ in these coordinates, one can obtain the following equations on every surface of $B_{\tau_0}=\{\tau=\text{constant }\tau_0\}$ :
\begin{align}
\label{eqn1}
&(\partial^2_{\tau}+z^2\Delta-2\psi^k\partial_k\partial_{\tau}+\psi^k\psi^l\partial^2_{kl})\gamma_{ij}=-2z^2(Ric_g)_{ij}+Q_{ij}(\gamma,\partial \gamma),\\
\label{eqn2}
&\Delta z+|A_{\gamma}|^2z+z Ric_{g}(\mathbf N,\mathbf N)-\psi(H_{\gamma})=0,\\
\label{eqn3}
&\Delta\psi^i+2z\langle D^2x^i,A_{\gamma}\rangle+z\partial_iH_{\gamma}+2[(A_{\gamma})^i_j\nabla^jz-\frac{1}{2}H\nabla^iz]+2z Ric_{g}(\mathbf N)^{i}=0.
\end{align}
Here the Laplacian operator $\Delta$ and covariant derivative $\nabla$ are with respect to the induced metric $\gamma$ on the level surface. 
In equation $(4.13)$, $Q_{ij}(\gamma,\partial \gamma)$ is a term which involves at most first order derivatives of $(\gamma, z,\psi)$ in all directions and the 2nd order derivatives of $z$ and $\psi$ along the directions tangent to the level surfaces. 
In equations (4.14) and (4.15), $\mathbf N$ denotes the normal vector of the surface $\{\tau=\text{constant}\}\subset C$, which is equal to 
\begin{equation}
\mathbf N=\frac{1}{z}(\partial_{\tau}-\psi).
\end{equation}
The second fundamental form $A_{\gamma}$ is given by 
\begin{equation}
A_{\gamma}=\frac{1}{2}\mathcal L_{\mathbf N}\gamma.
\end{equation} 
In addition, on the vertical boundary $\partial C=I\times S^1$, we have the following conditions,
\begin{equation}
z|_{\partial C}\equiv1,\psi|_{\partial C}\equiv 0.
\end{equation}

Now on the manifold $(S,~\tilde g)$, take an embedded cylinder $C\cong I\times B^2$ in such a manner that the horizontal boundary  $\partial_0C$  is embedded in $\partial S$, and the vertical boundary $\partial C$ is located in the interior of $S$. Equip $C$ with the induced metric, still denoted as $\tilde g$. Without loss of generality, we can assume the cylinder $C$ is sufficiently small so that $\tilde g$ is $C^{m,\alpha}$ close to the standard flat metric on the cylinder. Then we have the following local result.
\begin{proposition}
Let data $(\tilde g,\tilde u,\tilde \phi)$ be a stationary vacuum solution, i.e. $\Phi(\tilde g,\tilde u,\tilde \phi)=0$ in $C$. If $(k,w,\zeta)$ is an infinitesimal deformation of $(\tilde g,\tilde u,\tilde \phi)$ such that it solves the boundary value problem $(4.2)$ on $C$ in the sense that,
\begin{equation}
\begin{split}
\begin{cases}
D\Phi(k,w,\zeta)=0\\
\delta k=0
\end{cases}\text{on } C,~\quad
\begin{cases}
k^T=0\\
(A'_k)^T=0\\
w=\zeta=0\\
\mathbf n(w)+\mathbf n'_k(\tilde u)=0\\
\mathbf n(\zeta)+\mathbf n'_k(\tilde\phi)=0
\end{cases} \text{on } \partial_0 C,
\end{split}
\end{equation}
then there exists a vector field $X$ with $X=0$ on $\partial_0C$, such that 
$$k=\delta^*X,~w=L_X\tilde u,~\text{and}~\zeta=L_X\tilde\phi.$$
\end{proposition}
\begin{proof}
Define a Banach space $\mathcal M^*$ as follows, 
\begin{equation}
\resizebox{.9\textwidth}{!}
{$
\begin{split}
\mathcal{M}^*=\{&(g,u,\phi)\in [Met_{\delta}^{m,\alpha}\times C^{m,\alpha}_{\delta}\times C^{m,\alpha}_{\delta}](C):\delta_{\tilde g}g=0~\text{on}~\partial C,\\
&\big(g^{T},A_g,u,\phi,\mathbf n_g(u),\mathbf n_g(\phi)\big)=\big(\tilde g^{T},A_{\tilde g},\tilde u,\tilde \phi,\mathbf n_{\tilde g}(\tilde u),\mathbf n_{\tilde g}(\tilde \phi)\big)\text{ on }\partial_0C\}.
\end{split}
$}
\end{equation}
Obviously, $(\tilde g,\tilde u,\tilde\phi)\in\mathcal{M}^*$. By the hypothesis, the deformation $(k,w,\zeta)$ is tangent to the space $\mathcal M^*$, i.e. $(k,w,\zeta)\in T\mathcal{M}^*$.  Thus we can assume $(k,w,\zeta)$ is the infinitesimal deformation of a smooth curve $(g_t,u_t,\phi_t)$ at $t=0$, where $(g_t,u_t,\phi_t)\in\mathcal{M}^*$ for $t\in (-\epsilon,\epsilon)$, with some $\epsilon>0$, and $(g_0,u_0,\phi_0)=(\tilde g,\tilde u,\tilde\phi)$.

According to [AH], there exists a {smooth} curve of $C^{m+1,\alpha}$ diffeomorphisms $\Psi_t$ of $C$, which equal to $Id_{\partial_0C}$ on $\partial_0C$ for all $t\in(-\epsilon,\epsilon)$ and $\Psi_0=Id$ in $C$, so that $\Psi_t^*(g_t)$ share the same H-harmonic coordinates. We denote the infinitesimal variation of the new curve $(\Psi^*_t(g_t),\Psi^*_t(u_t),\Psi^*_t(\phi_t))$ at $t=0$ as $(g',u',\phi')$. It is given by
$$(g',u',\phi')=(k+\delta_{\tilde g}^*X,w+L_X\tilde u,\zeta+L_X\tilde \phi),$$
for some vector field $X$, with $X=0$ on $\partial_0C$. Therefore, to prove the proposition, it suffices to prove that $g'=u'=\phi'=0$.

Notice that since the diffeomorphism $\Psi^*_t$ is equal to $Id_{\partial_0C}$, it preserve the boundary conditions in (4.20). Thus the new curve $(\Psi^*_t(g_t),\Psi^*_t(u_t),\Psi^*_t(\phi_t))$ still belongs to the space $\mathcal M^*$ and consequently, the infinitesimal deformation $(g',u',\phi')$ still satisfies the boundary conditions listed in (4.19).

For simplicity of notation, the normalized curve $(\Psi^*_t(g_t),\Psi^*_t(u_t),\Psi^*_t(\phi_t))$ will still be denoted as $(g_t,u_t,\phi_t)$ in the following {argument}. Since  the infinitesimal deformation $(g',u',\phi')$ is the sum of a vacuum deformation $(k,w,\zeta)$, cf.(4.12), and a diffeomorphism deformation $\frac{d}{dt}\Psi_t^*$, it must preserve the stationary vacuum property, i.e.
$$\frac{d}{dt}|_{t=0}(\mathbf E,\mathbf F,\mathbf H)[(g_t,u_t,\phi_t)]=0\quad\text{in}\quad C.$$
This furthermore implies that,
\begin{align}
\label{eqn1}
s'_{g_t}=(2|du_t|^2+2e^{-4u_t}|d\phi_t|^2)',\\
\label{eqn2}
Ric'_{g_t}=(2du_t\otimes du_t+2e^{-4u_t}d\phi_t\otimes d\phi_t)',\\
\label{eqn3}
(\Delta u_t-2e^{-4u_t}|d\phi_t|)'=0,\\
\label{eqn4}
(\Delta\phi_t+4\langle du_t,d\phi_t\rangle)'=0,
\end{align}
where the prime superscript $'$ means $\frac{d}{dt}|_{t=0}$.

Let $\{\tau,x^i\}(i=1,2)$ denote the common H-harmonic coordinates for $g_t$, with the lapse function denoted as $z_t$ and the shift vector $\psi_t$. Thus the metric $g_t$ is in the form, 
$$g_t=z_td\tau^2+\gamma_t(\psi_t^id\tau+dx^i)(\psi_t^jd\tau+dx^j).$$
Write $(\gamma',z',\psi',u',\phi')$ as the infinitesimal variation of the curve $(\gamma_t,z_t,\psi_t,u_t,\phi_t)$ at $t=0$, then by the boundary conditions in $(4.19)$, we obtain 
\begin{equation}
\begin{cases}
\gamma'=0\\
(A'_{g'})^T=0\\
u'=0\\
\phi'=0\\
\mathbf n(u')+\mathbf n'(u_0)=0\\
\mathbf n(\phi')+\mathbf n'(\phi_0)=0,
\end{cases}
\text{ on }\partial_0C.
\end{equation}
In the above we use $(A'_{g'})$ to denote the variation of the second fundamental form $A$ at $g_0$ with respect to the infinitesimal deformation $g'$.

Moreover, equations (4.13-15) hold for all $(\gamma_t,z_t,\psi_t,u_t,\phi_t),~t\in(-\epsilon,\epsilon)$.
Linearization of the equation (4.14) at $t=0$ gives
\begin{equation}
\begin{split}
&\Delta_{\gamma_0} z'+|A_{\gamma_0}|^2z'+z'Ric_{g_0}(\mathbf N_0,\mathbf N_0)-\psi'(H_{\gamma_0})\\
=&-\Delta_{\gamma_t}' z_0-(|A_t|^2)'z_0-z_0[Ric_{g_t}(\mathbf N_t,\mathbf N_t)]'+\psi_0(H_t').
\end{split}
\end{equation}
Linearization of equation (4.15) at $t=0$ gives,
\begin{equation}
\resizebox{.99\textwidth}{!}
{$
\begin{split}
&\Delta_{\gamma_0} (\psi')^i+2z'\langle D^2_{\gamma_0}x^i,A_{\gamma_0}\rangle+z'\partial_iH_{\gamma_0}+2(A_{\gamma_0})^i_j\nabla_{\gamma_0}^jz'-H_{\gamma_0}\nabla_{\gamma_0}^iz'+2z'Ric_{g_0}(\mathbf N_0)^i\\
=&-(\Delta_{\gamma_t})'\psi_0^i-2z_0\langle D^2_{\gamma_t}x^i,A_{\gamma_t}\rangle'-z_0\partial_iH_{\gamma_t}'-\big(2(A_{\gamma_t})^i_j\nabla_{\gamma_t}^j\big)'z_0+\big(H_{\gamma_t}\nabla_{\gamma_t}^i\big)'z_0-2z_0(Ric_{g_t}(\mathbf N_t)^i)'.
\end{split}
$}
\end{equation}
The system (4.26)-(4.27) is a coupled elliptic system in the pair $(z',\psi')$ on every surface $B_{\tau}$ of constant $\tau$. Moreover, since (4.18) holds for all $(z_t,\psi_t)$, we have the following boundary conditions 
\begin{equation}
z'|_{\partial B_{\tau}}=0,~\psi'|_{\partial B_{\tau}}=0.
\end{equation}
Now since $g_{0}$ is assumed to be $C^{m,\alpha}$ close to the flat cylinder metric, there is a unique solution to the elliptic boundary value problem (4.26-28). Notice that based on equation (4.22), the linearized Ricci tensor $(Ric_{g_t})'$ that appears in (4.26) and (4.27) can be converted as,
\begin{equation}
\begin{split}
[Ric_{g_t}(\mathbf N_t,\mathbf N_t)]'&=Ric_{g_t}'(\mathbf N_0,\mathbf N_0)+2Ric_{g_0}(\mathbf N_t',\mathbf N_0)\\
&=(2du_t\otimes du_t+2e^{-4u_t}d\phi_t\otimes d\phi_t)'(\mathbf N_0,\mathbf N_0)\\
&\quad+2(2du_0\otimes du_0+2e^{-4u}d\phi_0\otimes d\phi_0)(\mathbf N_t',\mathbf N_0)\\
&=[2\big(\mathbf N_t(u_t)\big)^2+2e^{-4u_t}\big(\mathbf N_t(\phi_t)\big)^2]'.
\end{split}
\end{equation}
Combining the facts above, we conclude that the solution $(z',\psi')$ to (4.26-28) is uniquely determined by $(\gamma',u',\phi')$ on $B_{\tau}$ and their time derivatives $(\partial_{\tau}\gamma',\partial_{\tau} u',\partial_{\tau}\phi')$. Similar argument as in [AH] (Lemma 3.6) shows this also holds for the time derivatives $(\partial_{\tau}z',\partial_{\tau}\psi')$.

In particular, on the boundary surface $B_0=\partial_0C$, the terms $\Delta',(|A_t|^2)',$ and $H_t'$ in the second line of (4.26) all vanish because they only involve the tangential variation of $\gamma$ and $A$ which are zero on $\partial_0C$ according to $(4.25)$. Moreover we have $\mathbf N_t=-\mathbf n_{g_t}$ on the horizontal boundary $\partial_0 C$. So in the last line of (4.29), $[\mathbf N_t(u_t)]'=-\mathbf n (u')-\mathbf n'(u_0)$ on $\partial_0 C$ and it is zero by (4.25). The same holds for $[\mathbf N_t(\phi_t)]'$. Thus $[Ric_{g_t}(\mathbf N_t,\mathbf N_t)]'=0$ on $\partial_0 C$, and we can reduce equation (4.26) to the following one on the horizontal boundary $\partial_0 C$:
\begin{equation}
\Delta_{\gamma_0} z'+|A_{\gamma_0}|^2z'+z'Ric_{g_0}(\mathbf n,\mathbf n)-\psi'(H_{\gamma_0})=0.
\end{equation}
For the same reason, on $\partial_0C$ equation (4.27) can be simplified as,
\begin{equation}
\Delta_{\gamma_0}(\psi')^i+2z'\langle D_{\gamma_0}^2x^i,A_{\gamma_0}\rangle+z'\partial_iH_{\gamma_0}+2[(A_{\gamma_0})^i_j\nabla^jz'-\frac{1}{2}H_{\gamma_0}\nabla^iz']-2z'Ric_{g_0}(\mathbf n)^{i}=0.
\end{equation}
On the boundary of the surface $\partial_0C$, it follows from (4.28) that 
\begin{equation}
z'|_{\partial (\partial_0C)}=0,\psi'|_{\partial (\partial_0C)}= 0.
\end{equation}
As mentioned above, since $g_{0}$ is assumed to be $C^{m,\alpha}$ close to the flat metric, the boundary value problem (4.30-32) on $\partial_0C$ has a unique solution. One easily observe that it must be the trivial solution,
\begin{equation}
(z',\psi')=0\quad\text{on}~\partial_0C.
\end{equation}
Based on $(4.16)$ and $(4.17)$, we have 
\begin{equation*}
\begin{split}
A'_{\gamma}&=\frac{d}{dt}|_{t=0}[\frac{1}{2}\mathcal L_{\frac{1}{z_t}(\partial_\tau-\psi_t)}\gamma_t]\\
&=\frac{1}{2}\mathcal L_{\frac{1}{z_0}(\partial_\tau-\psi_0)}\gamma'+\frac{1}{2}\mathcal L_{\frac{1}{z_t}(-\psi')}\gamma_0+\frac{1}{2}\mathcal L_{-\frac{1}{z_0^2}z'(\partial_\tau-\psi_0)}\gamma_t.
\end{split}
\end{equation*}
From (4.25), we have $\gamma'=0$ and $(A'_{\gamma})^T=0$ on $\partial_0 C$. From (4.33), $z'=0, \psi'=0$ on $\partial_0 C$. Plugging these into the equation above, we can obtain,
\begin{equation}
\partial _{\tau}\gamma'_{ij}=0\quad\text{on}~\partial_0C.
\end{equation}
Since the unit normal vector $\mathbf N_t$ of $\partial_0 C$ as the slice $\{\tau=0\}$ and the geometric (outward) normal vector $\mathbf n_{g_t}$ of $\partial_0 C\subset (C, g_t)$ are related by $\mathbf N_t=-\mathbf n_{g_t}$ on $\partial_0C$, it follows that
$$\mathbf n'(u_0)=-[\frac{1}{z_t}(\partial_\tau-\psi_t)]'(u_0)=-[-\frac{1}{z_0^2}z'(\partial_\tau-\psi_0)+\frac{1}{z_0}(-\psi')](u_0)=0\quad\text{on}~\partial_0C.$$
Insert this to the fifth equation in (4.25), we obtain $\mathbf n(u')=0$ which further implies,
\begin{equation}
\partial_{\tau}u'=0\quad\text{on}~\partial_0C.
\end{equation}
Similarly, one can derive that,
\begin{equation}
\partial_{\tau}\phi'=0\quad\text{on}~\partial_0C.
\end{equation}
By the conditions in $(4.25)$ and $(4.34-36)$, the triple $(\gamma',u',\phi')$ has trivial Cauchy data on the boundary $\partial_0C$. In the interior of $C$, linearization of the equation $(4.13)$ shows
\begin{equation}
(\partial_{\tau}^2+z_0^2\gamma_0^{kl}\partial^2_{kl}-2\psi_0^k\partial_k\partial_{\tau}+\psi_0^k\psi_0^l\partial^2_{kl})\gamma'_{ij}=O(\gamma',z',\psi',u',\phi'),
\end{equation}
where $O(\gamma',w',\sigma',u',\phi')$ only depends on the tangential derivatives (at most 2nd order ) of $z',\psi'$, time derivative (at most 1st order) of $z',\psi'$ and derivatives (at most 1st order) of $\gamma',u',\phi'$. In addition, the analysis about the boundary value problem (4.26-28) shows that $(z',\psi')$ is uniquely determined (up to 2nd order tangential derivative and 1st order time derivative) by derivatives (at most 1st order) of $(\gamma',u',\phi')$. Thus equation (4.37) can be rewritten as, 
\begin{equation}
(\partial_{\tau}^2+z_0^2\gamma_0^{kl}\partial^2_{kl}-2\psi_0^k\partial_k\partial_{\tau}+\psi_0^k\psi_0^l\partial^2_{kl})\gamma'_{ij}=O(\gamma',u',\phi'),
\end{equation}
Similarly, equations $(4.23)$ and $(4.24)$ gives:
\begin{align}
\label{eqn1}
g^{\alpha\beta}\partial^2_{\alpha\beta} u'=O(\gamma',u',\phi'),\\
\label{eqn2}
g^{\alpha\beta}\partial^2_{\alpha\beta} \phi'=O(\gamma',u',\phi').
\end{align}
which are equivalent to the following equations,
\begin{align}
\label{eqn1}
[\partial^2_{\tau}-2\psi_0^{i}\partial_{i}\partial_{\tau}+(z_0^2\gamma_0^{ij}+\psi_0^i\psi_0^j)\partial^2_{ij}] u'=O(\gamma',u',\phi'),\\
\label{eqn2}
 [\partial^2_{\tau}-2\psi_0^{i}\partial_i\partial_{\tau}+(z_0^2\gamma_0^{ij}+\psi_0^i\psi_0^j)\partial^2_{ij}] \phi'=O(\gamma',u',\phi'),
\end{align}
since $g^{00}=z_0^{-2},g^{0i}=-z_0^{-2}\psi_0^{i}$, and $g^{ij}=\gamma_0^{ij}+z_0^{-2}\psi_0^i\psi_0^{j}$, where $i,j=1,2$ denote directions tangent to level surfaces $B_{\tau}$, and the superscript index $0$ denotes the $\partial_{\tau}$ direction.

Observe that equations $(4.38)$ and $(4.41-42)$ have the same principal operator. We denote it as $P$,
$$P=[\partial^2_{\tau}-2\psi_0^{i}\partial^{2}_{0i}+(w^2\gamma^{ij}+\psi_0^i\psi_0^j)\partial^2_{ij}].$$ 
This is the same as the operator discussed in [AH]. It is shown (cf.[AH] proof of Theorem 3.1 and 1.1) equations (4.38)-(4.41-41) together with vanishing Cauchy boundary conditions (4.25)-(4.34-36) have only the zero solution $(\gamma',u',\phi')=0$. This further implies that $(z',\psi')=0$ on $C$ and thus $(g',u',\phi')=0$. This completes the proof.

\end{proof}
Proposition 4.4 implies that there exists a vector field $Z$, which is zero on $\partial S$, such that $k=\delta_{\tilde g}^*Z,w=L_{Z}\tilde u,\zeta=L_{Z}\tilde \phi$ in a neighborhood $U$ of $\partial S$. Since in our case $S\cong\mathbb R^3-B$, $Z$ can be uniquely extended to a vector field on $S$ so that $k=\delta^*Z$ holds globally (cf. [A1]). 

Next we show that $Z$ must be zero. From the second equation in $(4.2)$, it follows that, 
$$\delta\delta^*Z=\delta k=0.$$ 
For a fixed $R>1$, let $B_R\subset S$ denote the closed ball of radius $R$ and $A_{\epsilon}$ denote the annulus between $B_{R-\epsilon}$ and $B_{R}$. Take a cutoff function $f\in C^{m+1,\alpha}(S)$ such that $f|_{B_{R-\epsilon}}\equiv 1$ and $f|_{S\setminus B_{R}}\equiv 0$. Let $W$ be the compactly supported vector field $W=fZ$. 
Since $Z$ is bounded in $B_{R}$, we can take $\epsilon$ small enough such that, 

\begin{equation}
\int_S\langle W,Z\rangle=\int_{B_{R-\epsilon}}|Z,Z|^2+\int_{A_{\epsilon}}\langle fZ,Z\rangle\geq\frac{1}{2}\int_{B_{R/2}}|Z|^2
\end{equation}
According to Lemma2.5, the map $\delta\delta^*$ is surjective, therefore there exist a vector field $Y\in T^{m,\alpha}_{\delta+4}$ that vanishes on $\partial S$ and makes
$$\delta\delta^*Y=W\text{ on }S.$$  
Notice that $\delta^*Z$ has the decay rate as $k~(\sim r^{-\delta})$. From this one can derive that  $Z$ can blow up no faster than $r^{2-\delta}$ (cf. \S 5.5). Therefore, applying integration by parts to the left side of equation (4.43), one can obtain
\begin{equation}
\begin{split}
&\int_S\langle W,Z\rangle=\int_S\langle \delta\delta^*Y,Z\rangle\\
=&\int_{S}\langle Y,\delta\delta^*Z\rangle+\big(\int_{\partial S}+\int_{\partial S_{\infty}}\big)[\delta^*Z(\mathbf n,Y)-\delta^*Y(\mathbf n, Z)]\\
=&0.
\end{split}
\end{equation} 
In the second line above, the boundary integral over $\partial S$ is zero because $Y=Z=0$ on $\partial S$ and the asymptotical behavior of $Y$ and $Z$ makes the boundary integral at infinity vanish. Combining equations $(4.43)$ and $(4.44)$, it is easy to derive that $Z=0$ in $B_{R/2}$, thus $k$, $w$, and $\zeta$ vanish in $B_{R/2}$ , which further implies that they are vanishing globally because of ellipticity of the system (4.2). This finishes the proof of Proposition 4.3. 
\begin{flushright}
$\Box$
\end{flushright}
In conclusion, we obtain the following result:
\begin{theorem}
The moduli space $\mathcal{E}_C$ is an infinite dimensional $C^{\infty}$ Banach manifold, with tangent space
$$T_{[(\tilde g,\tilde u,\tilde\phi)]}\mathcal{E}_C\cong Ker(D\Phi_{(\tilde g,\tilde u,\tilde\phi)}).$$ 
\end{theorem}
\begin{proof}
This is an immediate consequence of Theorem 4.2, the fact from \S2.3 that $\Phi^{-1}(0)=\mathcal E_C$ (locally), and the implicit function theorem in Banach spaces.

\end{proof}
Moreover, from the ellipticity results in \S3, it follows that,
\begin{theorem}
The boundary map,
\begin{equation*}
\begin{split}
&\quad\Pi:\mathcal E_C\rightarrow [S_2^{m,\alpha}\times C^{m-1,\alpha}\times C^{m-1,\alpha}](\partial S)\\
&\Pi[(g,u,\phi)]=(e^{-2u}g^T,~e^u(H_{g}-2\mathbf n_g(u)),~e^u\mathbf n_{g}(\phi))
\end{split}
\end{equation*}
is a $C^{\infty}$ Fredholm map, of Fredholm index 0.
\end{theorem}
\begin{proof}
The fact from \S3.2 that the operator $P_2$ is elliptic implies that the boundary map $\tilde\Pi$
\begin{equation*}
\begin{split}
&\quad\tilde\Pi:~\mathbf Z_C\rightarrow [S_2^{m,\alpha}\times C^{m-1,\alpha}\times C^{m-1,\alpha}](\partial S),\\
&\tilde\Pi(g,u,\phi)=
(e^{-2u}g^T,~e^u(H_{g}-2\mathbf n_g(u)),~e^u\mathbf n_{g}(\phi))
\end{split}
\end{equation*}
is smooth and Fredholm. It is of Fredholm index 0 because $P_2$ is formally self-adjoint. Moreover, since we show in \S 2.3 that locally $\mathcal E_C=\mathbf Z_C$, it follows that $\Pi$ is also a smooth Fredholm map and of index 0.

\end{proof}
Now translating the results above from conformal data $(g,u,\phi)$ back to $(g_S,u,\phi)$ via the isomorphism $\mathcal Q$ as in \S3.3, gives Theorem 1.2.\\

$Remark.$ All the methods and results in this paper can be applied equally well to the interior problem where $S\cong B^3$.
\section{Appendix}
In this section, we provide the detailed computation of the linearization of operator $\mathcal{P}$, the linearization of reduced Hilbert-Einstein functional $I$, and some other basic results used in this paper. We refer to [Be] for elementary formulas of the differentials of various geometric tensors.
\subsection{Linearization of the interior operator $\mathcal L$}$~~$

Let $h$ be the infinitesimal deformation of the metric $g$ on $S$.  Then the resulting variation of Ricci tensor is given by,
\begin{equation}
2Ric'_{h}=D^{\ast}Dh-2\delta^{\ast}\delta h-D^2(trh)+O_0.
\end{equation}
\\The variation of scalar curvature is given by,
\begin{equation}s'_h=\Delta_g(trh)+\delta\delta h+O_0.
\end{equation}

Linearization of the gauge term $\delta^*G$ in operator $\mathcal L$ are as follows. 
\\For the Bianchi gauge, we have,
\begin{equation}
\begin{split}
2[\delta^{\ast}G_1]'_h&=2[\delta^{\ast}\beta_{\tilde g}(g)]'_h\\
&=L_{\beta_{\tilde g}g}h+2\delta^{\ast}\beta_{\tilde{g}}h\\
&=L_{\beta_{\tilde g}g}h+2\delta^{\ast}\delta_{\tilde{g}}h+D^2(trh).
\end{split}
\end{equation}
For the divergence gauge, we have,
\begin{equation}
\begin{split}
2[\delta^{\ast}G_2]'_h&=2[\delta^{\ast}\delta_{\tilde{g}}g]'_h=L_{\delta_{\tilde g}g}h+2\delta^{\ast}\delta_{\tilde{g}}h.
\end{split}
\end{equation}

Combining equations $(5.1)$ and $(5.3)$, one can derive the linearization for $\mathcal L$, with the Binachi gauge, at $\tilde g=g$:
$$L_1(h)=D^*Dh+O_0.$$ 

Combining equations $(5.1-2)$ and $(5.4)$, one can derive the linearization for $\mathcal L$, with the divergence gauge, at $\tilde g=g$:
$$L_2(h)=D^*Dh-D^2(trh)-(\Delta_g(trh)+\delta\delta h)g+O_0.$$ 
\subsection{Linearization of the boundary operator $\mathcal{B}$}$~$

We know that the normal vector $\mathbf{n}$ of $\partial S$ satisfies the following equations,
$$\begin{cases}
g(\mathbf n,\mathbf n)=1,\\
g(\mathbf n,T)=0,
\end{cases}$$
where $T$ is a tangential vector.
Let $\mathbf{n}'_{h}$ denote the variation of $\mathbf{n}$ with respect to deformation $h$. Then linearization of the above equations gives,
$$\begin{cases}
2g(\mathbf{n},\mathbf{n}'_h)+h(\mathbf{n},\mathbf{n})=0,\\
g(\mathbf{n}'_h,T)+h(\mathbf{n},T)=0,
\end{cases}$$
from which, one can solve for the term $\mathbf{n}'_{h}$ as,
\begin{equation}
\mathbf{n}'_h=-\frac{1}{2}h(\mathbf{n},\mathbf{n})\mathbf{n}-h(\mathbf{n})^{T}=-h(\mathbf n)+\frac{1}{2}h(\mathbf n,\mathbf n)\mathbf n.
\end{equation}

The variation $H'_h$ of mean curvature $H_g$ is given by 
\begin{equation}
2H'_h= 2(trA)'_h=2trA'_h-2\langle A_g, h\rangle,
\end{equation}
where $A_g$ is the second fundamental form of $\partial S\subset (S,g)$, defined by $A_g=\frac{1}{2}L_{\mathbf{n}}g$.
Linearization of $A$ is as follows, 
\begin{equation*}
2A'_h=(L_{\mathbf{n}}h+L_{\mathbf{n}'_h}g)=\nabla_{\mathbf{n}}h+2h\circ\nabla \mathbf n+L_{\mathbf{n}'_h}g.
\end{equation*}
Taking the trace of the equation above, we obtain,
\begin{equation*}
2trA'_h=\nabla_{\mathbf{n}}trh+2\langle h,A\rangle-2\delta(\mathbf{n}'_h).
\end{equation*}
Pluging the expression $(5.5)$ for $\mathbf n'_h$ into the equation above, we obtain,
\begin{equation}
\begin{split}
2trA'_h&=\nabla_{\mathbf{n}}trh+2\langle h,A\rangle+2\delta^T(h(\mathbf n)^T)-\mathbf n(h(\mathbf n,\mathbf n))+O_0\\
&=\nabla_{\mathbf{n}}tr^Th+2\langle h,A\rangle+2\delta^T(h(\mathbf n)^T)+O_0.
\end{split}
\end{equation}
Combining equations $(5.6)$ and $(5.7)$ gives,
$$H'_h=\frac{1}{2}\nabla_0(h_{11}+h_{22})-\Sigma_{k=1}^{2}\nabla^k(h_{0k})+O_0,$$
which is the same as used in the symbol computation in \S3.1.
\subsection{Variation of the functional $I$}$~$

First, we define a functional $\tilde I$ as,
$$\tilde I=\int_{S} s_g-2|du|^2-2e^{-4u}|d\phi|^2 dvol_g.$$
Since the variation of scalar curvature $s_g$ is given by,
$$s'_h=\Delta_g(trh)+\delta\delta h-\langle Ric_g, h\rangle,$$
linearization of $I$ with respect to the metric is as follows,
\begin{equation*}
\begin{split}
\tilde I_g'(h)&=\int_S [\Delta_g(trh)+\delta\delta h-\langle Ric_g, h\rangle+2h(du,du)+2e^{-4u}h(d\phi,d\phi)\\
&\quad\quad\quad+\frac{1}{2}trh(s_g-2|du|^2-2e^{-4u}|d\phi|^2)] \\
&=\int_S [\Delta_g(trh)+\delta\delta h\\
&\quad\quad\quad+\langle -Ric_g+2du\otimes du+2e^{-4u}d\phi\otimes d\phi+\frac{1}{2}(s_g-2|du|^2-2e^{-4u}|d\phi|^2)g,h\rangle]\\
&=\int_S [\Delta_g(trh)+\delta\delta h+\langle \mathbf E,h\rangle].
\end{split}
\end{equation*}
The term $\int_S [\Delta_g(trh)+\delta\delta h]$ in the expression above can be converted to a boundary term as,
\begin{equation}
\begin{split}
&\int_S [\Delta_g(trh)+\delta\delta h]\\
&=-\int_{\partial S} [\mathbf n(trh)+\delta h(\mathbf n)]-\int_{\partial S_{\infty}} [\mathbf n(trh)+\delta h(\mathbf n)]\\
&=-\int_{\partial S} [\mathbf n(trh)-\mathbf n(h_{00})+\langle h,A\rangle]-\int_{\partial S_{\infty}} [\mathbf n(trh)+\delta h(\mathbf n)]. 
\end{split}
\end{equation}
To balance the finite boundary term, we add an extra term $I_B=\int_{\partial S}2H_g$ to the functional $I$. Notice that the first variation of $I_B$ with respect to the metric is given by,
\begin{equation*}
\begin{split}
I_B'(h)&=\int_{\partial S}[2H'_h+tr^ThH_g] \\
&=\int_{\partial S} [\mathbf n(trh)-2\delta(\mathbf n'_h)+tr^ThH_g] .
\end{split}
\end{equation*}
For a generic vector field $V$ on $\partial S$, we have $\delta V=\delta^T V^T-\mathbf n(V_0)$. Thus, we can simplify the term $\delta(\mathbf n'_h)$ in the boundary term above and obtain,
\begin{equation}
\begin{split}
I_B'(h)=\int_{\partial S} [\mathbf n(trh)-\mathbf n(h_{00})+tr^ThH_g] .
\end{split}
\end{equation}
Combining equations $(5.8)$ and $(5.9)$ we can obtain the formulae of variation for the functional $\tilde I+I_B$:
\begin{equation*}
(\tilde I+I_B)'(h)=\int_S \langle\mathbf E,h\rangle  +\int_{\partial S}[-\langle A,h\rangle+tr^ThH_g] -\int_{\partial S_{\infty}} [\mathbf n(trh)+\delta h(\mathbf n)].
\end{equation*}
To remove the boundary term at infinity, we use the mass $m_{\text{ADM}}(g)$, as shown in equation $(3.25)$.
\subsection{Formal self-adjointness of $D\hat\Phi$}$~$

To prove the self-adjointness for $D\hat\Phi$, we will use the functional $I$, as defined in \S3, 
$$I=\int_{S} s_g-2|du|^2-2e^{-4u}|d\phi|^2 dvol_g+2\int_{\partial S}Hdvol_{g^T}+16\pi m_{ADM}(g).$$
Recall from \S3, the first variation of $I$ is given by,
\begin{equation}
\begin{split}
I'_{(\tilde g,\tilde u,\tilde \phi)}(h,v,\sigma)&=\int_S \langle (\mathbf E, \mathbf F, \mathbf H), (h,v,\sigma)\rangle\\
&\quad\quad+\int_{\partial S}[ -\langle A_{\tilde g},h\rangle+H_{\tilde g}trh^{T}-4\mathbf n(u)v-4e^{-4u}\sigma \mathbf n(\phi)].
\end{split}
\end{equation}
Let $ (h,v,\sigma)$ be in the tangent space $T\mathcal M_S$, which is defined by
\begin{equation}
\begin{split}
T\mathcal M_S=\{~&(h,v,\sigma)\in [(S_2)_{\delta}^{m,\alpha}\times C^{m,\alpha}_{\delta}\times C^{m,\alpha}_{\delta}](S): \\
&\begin{cases} 
\delta_{\tilde g}h=0,\\
h^T-2v\tilde g^T=0,\\
[e^u(H_{\tilde g}-2\mathbf n(u))]'_{(h,v)}=0,\\
[e^{-2u}\mathbf n(\phi)]'_{(h,v,\sigma)}=0,
\end{cases}\quad\text{on } \partial S~\}.
 \end{split}
 \end{equation}
 Applying the boundary conditions in $(5.11)$ to the equation $(5.10)$, we obtain,
 \begin{equation*}
\begin{split}
&I'_{(\tilde g,\tilde u,\tilde \phi)}(h,v,\sigma)\\
&=\int_S \langle(\mathbf E, \mathbf F, \mathbf H), (h,v,\sigma)\rangle+\int_{\partial S} [-\langle A_{\tilde g},2v\tilde g\rangle+2vH_{\tilde g}tr\tilde g^{T}-4\mathbf n(u)v-4e^{-4u}\sigma \mathbf n(\phi)]\\
&=\int_S \langle (\mathbf E, \mathbf F, \mathbf H), (h,v,\sigma)\rangle+\int_{\partial S} [2v(H_{\tilde g}-2\mathbf n(u))-4e^{-4u}\sigma \mathbf n(\phi)].
\end{split}
\end{equation*}
Taking the variation of $I'$ with respect a deformation $(k,w,\zeta)\in T\mathcal M_S$, we obtain
 \begin{equation*}
\begin{split}
I''_{(\tilde g,\tilde u,\tilde \phi)}[(h,v,\sigma),(k,w,\zeta)]&=\int_S \langle D(\mathbf E, \mathbf F, \mathbf H){(k,w,\zeta)}, (h,v,\sigma)\rangle\\
&\quad+\int_{\partial S} \{2v[H_{\tilde g}-2\mathbf n(u)]'_{(k,w)}-\sigma[4e^{-4u}\mathbf n(\phi)]'_{(k,w,\zeta)}\}\\
&\quad+\int_{\partial S}\{\frac{1}{2}trk^T[2v(H_{\tilde g}-2\mathbf n(u))-4e^{-4u}\sigma \mathbf n(\phi)]\}.
\end{split}
\end{equation*}
According to $(5.11)$, we have $k^T=2w\tilde g^T$ and $ 2v[H_{\tilde g}-2\mathbf n(u)]'_{(k,w)}=-2vw(H_{\tilde g}-2\mathbf n(u))$. Thus the boundary terms in the expression above can be simplifies as
 \begin{equation*}
\begin{split}
&I''_{(\tilde g,\tilde u,\tilde \phi)}[(h,v,\sigma),(k,w,\zeta)]\\
&=\int_S \langle D(\mathbf E, \mathbf F, \mathbf H){(k,w,\zeta)}, (h,v,\sigma)\rangle+\int_{\partial S} 2wv[H_{\tilde g}-2\mathbf n(u)]'_{(k,w)}\\
&\quad+\int_{\partial S}\{-\sigma[4e^{-4u}\mathbf n(\phi)]'_{(k,w,\zeta)}-8we^{-4u}\sigma \mathbf n(\phi)\}\\
&=\int_S \langle D(\mathbf E, \mathbf F, \mathbf H){(k,w,\zeta)}, (h,v,\sigma)\rangle+\int_{\partial S} 2wv[H_{\tilde g}-2\mathbf n(u)]'_{(k,w)}\\
&\quad+\int_{\partial S}4e^{-4u}\sigma[2w\mathbf n(\phi)-(\mathbf n(\phi))'_{(k,\zeta)}].
\end{split}
\end{equation*}
Here the last boundary term vanishes, because 
$$4e^{-4u}\sigma[2w\mathbf n(\phi)-(\mathbf n(\phi))'_{(k,\zeta)}]=-4e^{-2u}[e^{-2u}\mathbf n(\phi)]'_{(k,w,\zeta)}=0,$$
based on the conditions in $(5.11)$.
Therefore, from the symmetry of the 2nd order variation of the functional $I$, it follows that,
 \begin{equation}
\begin{split}
\int_S \langle D(\mathbf E, \mathbf F, \mathbf H){(k,w,\zeta)}, (h,v,\sigma)\rangle=\int_S \langle D(\mathbf E, \mathbf F, \mathbf H){(h,v,\sigma)}, (k,w,\zeta)\rangle.
\end{split}
\end{equation}
In addition, it is easy to derive that 
 \begin{equation}
\begin{split}
\int_S \langle \delta^*\delta k,h\rangle =\int_S\langle \delta^*\delta h,k\rangle \quad\text{for}~k,h\in T\mathcal M_S.
\end{split}
\end{equation}
Combining equations $(5.12)$ and $(5.13)$, we obtain the formal self-adjointness for $D\hat\Phi$, i.e.
 \begin{equation*}
\int_S \langle D\hat\Phi[(k,w,\zeta)], (h,v,\sigma)\rangle=\int_S \langle D\hat\Phi[(h,v,\sigma)], (k,w,\zeta)\rangle,~\forall (k,w,\zeta),(h,v,\sigma)\in T\mathcal M_S.
\end{equation*}

\textit{Remark} The calculation here is basically the same as in the proof of Proposition 3.5. But there is still difference between them. In Proposition 3.5, we only consider the linearization of $\mathbf E$ with respect to the metric deformation $h$ (or $k$), and same for $\mathbf F$ and $\mathbf H$. This the reason that in (3.32), the linearizations are written as $(\mathbf E'_h,\mathbf F'_v,\mathbf H'_{\sigma})$. However, here we consider the variation of $\mathbf E$ in all directions of deformation $(h,v,\sigma)$ and also for $\mathbf F,\mathbf H$. So we are writting $D(\mathbf E, \mathbf F, \mathbf H){(k,w,\zeta)}$ in (5.12).
\subsection{The blow-up rate of $Z$}
Given that $\delta^*Z$ has the decay rate $\delta$ (denoted as $\delta^*Z\sim r^{-\delta}$), we will find an upper bound for the blow-up rate of $Z$ in the following. 

For a radius $r$ large enough, let $S_r\subset S$ be the sphere of radius $r$. Let $N_r$ denote the unit normal vector of the sphere pointing outwards, then we have
\begin{equation*}
\delta^* Z(N_r,N_r)=N_r[\tilde g(Z,N_r)]+\tilde g(Z,\nabla_{N_r}N_r).
\end{equation*}
One can extend $N_r$ in a way such that $\nabla_{N_r}N_r=0$, and thus 
$$N_r[\tilde g(Z,N_r)]\sim r^{-\delta}.$$
Therefore, $g(Z,N_r)$ blows up no faster than $r^{1-\delta}$.
\\Let $Z^T$ denote the tangential component of $Z$ along $S_r$, i.e. $Z^T=Z-g(Z,N_r)N_r$. Basic calculation shows
\begin{equation*}
\begin{split}
2\delta^* Z(N_r, Z^T)&=\tilde g(\nabla_{N_r}Z,Z^T)+\tilde g(\nabla_{Z^T}Z,N_r)\\
&=\tilde g(\nabla_{N_r}(Z^T+g(Z,N_r)N_r),Z^T)+Z^T[\tilde g(Z,N_r)]-\tilde g(Z,\nabla_{Z^T}N_r)\\
&=\tilde g(\nabla_{N_r}Z^T,Z^T)+Z^T[\tilde g(Z,N_r)]-A(Z^T,Z^T)\\
&=|Z^T| N_r(|Z^T|)+Z^T[\tilde g(Z,N_r)]-A(Z^T,Z^T),
\end{split}
\end{equation*}
where $A$ denotes the second fundamental form of the hypersurface $S_r\subset (S,\tilde g)$. Thus we obtain,
\begin{equation*}
N_r(|Z^T|)-|Z^T|A(\frac{Z^T}{|Z^T|},\frac{Z^T}{|Z^T|})=2\delta^* Z(N_r,\frac{Z^T}{|Z^T|})-\frac{Z^T}{|Z^T|}[\tilde g(Z,N_r)].
\end{equation*}
It is obvious that the right side of the equation above blows up no faster than  $r^{1-\delta}$. The left side is essentially equal to $\big(\partial_r(|Z^T|)-\frac{1}{r}|Z^T|\big)$ since $N_r$ is asymptotically $\partial_r$ and the second fundamental form $A$ is asymptotically equal to $\frac{1}{r}g$. Thus $\big(\partial_r(|Z^T|)-\frac{1}{r}|Z^T|\big)$ blows up no faster than $r^{1-\delta}$, which implies that the increasing rate of $|Z^T|$ is at most  $r^{2-\delta}$.

\bibliographystyle{amsplain}

\end{document}